\documentclass[12pt]{amsart}
\usepackage{amsmath,amscd,amsthm,amsfonts, amssymb,amsxtra, accents, ulem}
\usepackage[all,cmtip]{xy}
\usepackage{xparse,etoolbox}
\usepackage{tikz}
\usetikzlibrary{patterns,calc,decorations.markings}
\usepackage{color}
\usepackage{hyperref}
  \hypersetup{colorlinks=true,citecolor=blue}
  \usepackage{graphicx, epstopdf}

 \usepackage[font=scriptsize]{caption} 
\CDat
\newtheorem{thm}{Theorem}[section]  
\newtheorem*{un-no-thm}{Theorem}
\newtheorem{cor}[thm]{Corollary}     
\newtheorem{lem}[thm]{Lemma}         
\newtheorem{prop}[thm]{Proposition}  
\newtheorem{add}[thm]{Addendum}

\newtheorem{bigthm}{Theorem}

\newtheorem{bigadd}[bigthm]{Addendum}

\theoremstyle{definition}
\newtheorem{defn}[thm]{Definition}   

\theoremstyle{definition}
\newtheorem*{prob}{Problem}   

\theoremstyle{definition}
\theoremstyle{remark}
\newtheorem{rem}[thm]{Remark}
\newtheorem{rems}[thm]{Remarks}

\theoremstyle{definition}
\newtheorem{ques}[thm]{Question}

\theoremstyle{remark}

\newtheorem{notation}[thm]{Notation}
\newtheorem*{acks}{Acknowledgements}

\newtheorem*{out}{Outline}

\newtheorem*{intro-rem}{Remark}
\newtheorem*{intro-rems}{Remarks}

\theoremstyle{remark}
\newtheorem{ex}[thm]{Example}

\newcommand{\blocktheorem}[1]{%
  \csletcs{old#1}{#1}
  \csletcs{endold#1}{end#1}
  \RenewDocumentEnvironment{#1}{o}
    {\par\addvspace{1.5ex}
     \noindent\begin{minipage}{\textwidth}
     \IfNoValueTF{##1}
       {\csuse{old#1}}
       {\csuse{old#1}[##1]}}
    {\csuse{endold#1}
     \end{minipage}
     \par\addvspace{1.5ex}}
}

\blocktheorem{bigthm}

\newcommand{\Mdef}[2]{\newcommand{#1}{\relax\ifmmode #2 \else $#2$\fi}}

\def\:{\colon\!}
\Mdef{\smsh}{\wedge}
\def\Bbb{\mathbb}

\def\cal{\mathcal}

\def\an{\text{\rm an}}

\newcommand{\End}{\mathrm{end}}
\newcommand{\Ch}{\cal C}
\newcommand{\nt}{\mathrm{nt}}
\DeclareMathOperator{\mycolim}{colim}
\newcommand{\colim}{\mathop{\mycolim}}
\DeclareMathOperator{\myhocolim}{hocolim}
\newcommand{\hocolim}{\mathop{\myhocolim}}

\DeclareMathOperator{\Top}{Top}

\newcommand\restr[2]{{ 
  \left.\kern-\nulldelimiterspace 
  #1 
  \vphantom{\big|} 
  \right|_{#2} 
}}

\title{Hypercurrents}
\author[M.~J. Catanzaro]{Michael J.\ Catanzaro}
\address{Dept.~ of Mathematics\\ 
                Iowa State University \\
               Ames, IA 50011}
\email{mjcatanz@iastate.edu}

\author[V.~Y. Chernyak]{Vladimir Y.\ Chernyak}
\address{Dept.~  of Chemistry\\ 
                Wayne State University\\ 
                Detroit, MI 48202}
\email{chernyak@chem.wayne.edu}

\author[J.~R. Klein]{John R.\ Klein}
\address{Dept.~  of Mathematics\\
                Wayne State University\\
                Detroit, MI 48202}
\email{klein@math.wayne.edu}

\date{\today}
\begin{document}

\begin{abstract} We introduce the notion of a {\it protocol}, which 
consists of a space whose points are labeled by real numbers indexed 
by the set of cells of a fixed CW complex in prescribed degrees,  where the labels are required to vary continuously. 
If the space is a one-dimensional manifold, then a protocol determines a continuous time Markov chain.

When a homological gap condition is present, we associate to each protocol a `characteristic' cohomology class which we call the {\it hypercurrent.} The hypercurrent comes in two flavors: one algebraic topological and the other analytical. For generic protocols we show that the analytical hypercurrent tends to the topological hypercurrent in the low temperature limit. We also exhibit examples of protocols having nontrivial hypercurrent.
\end{abstract}
\maketitle
\setcounter{tocdepth}{1}
\tableofcontents
\addcontentsline{file}{sec_unit}{entry}


\section{Introduction}\label{sec:intro}
\subsection{Background} This paper is a result of our ongoing investigation of the stochastic motion of cellular cycles in a finite CW complex \cite{CKS13}, \cite{CCK15b}, \cite{CCK16}.
In \cite{CKS13}, we introduced the notion of a {\it driving protocol}, which is a one-parameter
 family of potentials labeling the vertices and edges of a connected finite topological graph $X$.  We explained how a
 driving
protocol determines a continuous time Markov chain whose state diagram is the double of $X$. 
Associated with such a Markov chain, we defined an {\it average current} which is a 1-dimensional chain with real coefficients that measures of the net flow of probability across the edges of $X$. When the protocol is periodic
of period $\tau$, we showed in the adiabatic slow driving limit $\tau\to \infty$ 
that the resulting current tends to a 1-dimensional real homology class. After 
taking a second limit, the `low temperature limit,' we showed that,  for generic parameters, 
the average current quantizes to an integer homology class.
Furthermore, we showed that the integer homology class admits an algebraic topological description.

In \cite{CCK16} we allowed $X$ to be  an arbitrary finite CW complex of dimension $d$.
In this case a driving protocol consists of a one-parameter family of potentials labeling the cells of $X$ in the contiguous dimensions $d-1$ and $d$. When the driving protocol is periodic, the average current in the adiabatic limit 
was defined as a real $d$-dimensional homology class. In this generality, we
extended the principal results of \cite{CKS13}. In particular, we showed that for generic parameters the average current in the low temperature limit fractionally quantizes to a homology class with coefficients in $\Bbb Z[\frac{1}{\delta}]$, for some positive integer $\delta$ that is a combinatorial invariant of $X$.

In the current paper we consider a further generalization: we introduce driving protocols in higher dimensions. 
We are tempted to conjecture that an $n$-dimensional protocol determines an
`$n$-dimensional Markov process' with time substituted by an $n$-manifold.
However, we are not in a position to state a conjecture since we do not know what an
`$n$-dimensional Markov process' should be.
Nevertheless, in the presence of a homological gap condition, we are able to 
associate a topological invariant to a higher dimensional protocol, called the {\it hypercurrent,} which is 
a kind of secondary current.   

We will exhibit several different descriptions of the hypercurrent,
some algebraic topological as well as an analytical one defined using differential forms. We axiomatically characterize
the hypercurrent and exhibit some examples. Lastly, we prove a quantization result relating the analytical
definition to the topological one; this a kind of {\it index theorem.}

\subsection{Motivation} We recall from \cite{CCK16} how continuous 
time Markov chains relate to driving
protocols.

\subsubsection{Markov chains}
A {\it continuous time Markov chain} consists of:
\begin{itemize}
\item a locally finite directed  graph
 \[
 \Gamma = (\Gamma_0,\Gamma_1)\, ,
 \]
called a {\it state diagram};
\item an assignment of 
a continuous function 
  \[ k_{\alpha}\: \Bbb R \to [0,\infty)\, ,
 \]
 to each edge $\alpha\in \Gamma_1$. 
\end{itemize}
The function $k_{\alpha}$  is called the {\it transition rate} of the edge $\alpha$. It is to be interpreted as
  the instantaneous rate of change of probability in jumping 
  from the source state of $\alpha$ to its target state.
  
 Let $s = (s_0,s_1)\: \Gamma_1\to \Gamma_0\times \Gamma_0$ be the function which assigns to $\alpha$ the ordered pair consisting of its
source and target.
Then the rates define a time dependent square matrix $\Bbb H=\Bbb H(t)$, as follows.
For $i\ne j$, set
\[
h_{ij} = \sum_{s(\alpha) = (i,j)} k_\alpha\, .
\]
Then the matrix entries of $\Bbb H$ are given by 
 \[
 \Bbb H_{ij} = \begin{cases} 
h_{ij}\, ,\qquad & i\ne j \, ;\\
  -\sum_{\ell\ne j} h_{\ell j} \, ,  \quad  \text{ if }  & i=j\, ,
 \end{cases}
 \] 
 where the indices range over $i,j\in \Gamma_0$.
 The  matrix $\Bbb H$ is called the {\it master operator}.  
 Associated with $\Bbb H$ is a linear, first order 
 ordinary differential equation
 \begin{equation} \label{eqn:kolmogorov}
 p'(t) = \Bbb Hp(t),\qquad 
\end{equation}
in which $p(t)$ is a one parameter family of (probability) distributions on the set of vertices $\Gamma_0$.  
Equation \eqref{eqn:kolmogorov} is called the {\it Kolmogorov equation} or the {\it master equation}. 
Its solutions describe the flow of probability.

Note that when the transition rates are constant with value 1, the matrix $\Bbb H$ is with the graph Laplacian  
and  \eqref{eqn:kolmogorov} is the heat (diffusion) equation.

\subsubsection{Driving protocols}
A convenient way to obtain a Markov chain is to fix  a finite, connected, simple undirected graph  
\[X = (X_0,X_1)\, .\]  
 Consider the real vector space
 $\cal M_{X}$ consisting of pairs $(E,W)$ in which $E\: X_0 \to \Bbb R$
 and $W\: X_1 \to \Bbb R$ are functions. Thus $E$ and $W$ equip the vertices and edges of the graph with real number {\it weights}.  For a vertex $i$, we let
 $E_i$ denote the value of $E$ at a vertex $i$, and 
 we let $W_{\alpha}$ denote the value of $W$ at an edge $\alpha$. Note that
 
As in \cite{CKS13}, a {\it driving protocol} is a continuous map 
 \[
 \gamma\: \Bbb R \to \cal M_{X}\, ,
 \] 
i.e.,  a one-parameter family of weights $(E(t),W(t))$.

 A driving protocol $\gamma$ determines a continuous time Markov chain whose  
 state diagram
 is the ``double'' of $X$.  The latter is
  the directed graph $\Gamma = (\Gamma_0,\Gamma_1)$ 
 in which $\Gamma_0 = X_0$ and
$\Gamma_1$ consists of the ordered pairs 
 $(i,\alpha)\in X_0 \times X_1$  such that $i$ is a vertex of $\alpha$.
 The transition rate at  $(i,\alpha)$  is defined by
\begin{equation} \label{eqn:Arrhenius}
k_{i\alpha}(t) :=  e^{E_i(t) -W_{\alpha}(t)}\, .
\end{equation}

\subsection{Generalized protocols} We now turn to the main object of this study.
Let $X$ be a finite CW complex and let $X_k$ denote its set of $k$-cells.
For natural numbers $p \le q$, let $\cal M_{p,q}(X)$ be the vector space
of functions ({\it weights})
\[
W_\bullet \: \coprod_{k=p}^q X_k \to \Bbb R\, .
\]
We let $W_j\: X_j \to \Bbb R$ denote the restriction of $W_\bullet$ to $X_j$.

\begin{defn} Let $\Sigma$ be a topological space.
A {\it $\Sigma$-protocol} is a  continuous map
\[
\gamma\:\Sigma \to \cal M_{p,q}(X)\, .
\] 
\end{defn}

\begin{rem} Hereafter, we refer to these as {\it protocols}. When $X$ is understood,
we abbreviate notation $\cal M_{p,q} = \cal M_{p,q}(X)$.  
There is an evident projection
\[
\cal M_{p,q} \to \cal M_{r,s}
\]
whenever $p \le r \le s \le q$.  Given a protocol $\gamma$ and a map of spaces $f\: \Sigma' \to\Sigma$,
 the composition $\gamma' := \gamma\circ f$ defines a protocol.
 
Given $\gamma$ and $b\in \Sigma$, we write $W_j(b)\: X_j \to \Bbb R$ for the
$j$-th component of $\gamma(b)$, i.e., 
\[
\gamma(b) = W_\bullet(b) := (W_p(b),W_{p+1}(b),\dots, W_q(b))\, .
\]
\end{rem}

\subsection{Homological hypercurrents}
As $\cal M_{p,q}$ is contractible, there is at this point no algebraic topology
that can be extracted from a protocol $\Sigma\to \cal M_{p,q}$. Algebraic
topology enters the picture by imposing restrictions on the kinds of families that are allowed.\footnote{This is akin to what one does with the space of smooth
functions $C^\infty(M,\Bbb R)$ on a compact smooth manifold $M$. The space  $C^\infty(M,\Bbb R)$
is contractible. Algebraic topology materializes
by restricting to a generic subspace of functions such as the space of
Morse functions.}
The first type of restriction that we will consider are  ``good'' protocols:

\begin{defn}[Good Protocols] \label{defn:good-family} A 
  protocol $\Sigma\to \cal M_{p,q}$
is {\it good} if for each $b\in \Sigma$, there exists an integer
 $j = j(b)$ satisfying $p\le j \le q$  such that the weight 
$W_j(b) \: X_j \to \Bbb R$ is one-to-one.\footnote{In \S\ref{sec:robust}, we relax the condition by considering so-called {\it robust} protocols.}

\end{defn}

Let \[
\cal {\breve M}_{p,q} \subset \cal M_{p,q}
\] be the (open, dense) subspace of $\cal M_{p,q}$ consisting of those weights $W_\bullet$
such that $W_j\: X_j \to \Bbb R$ is one-to-one for some $p \le j \le q$.
Then a good protocol is nothing more than a map 
$\Sigma \to \cal {\breve M}_{p,q}$. 
  
Collectively, the good protocols are objects of the
category  
\[
\Top_{/\breve {\cal M}_{p,q}}\, ,
\] i.e., the category of spaces over $\breve {\cal M}_{p,q}$. When the structure map is understood we 
drop it from the notation: an object
$\Sigma \to \breve {\cal M}_{p,q}$ is
then referred to as $\Sigma$.

 \begin{defn}[Gap Condition]
A {\it gap} in $X$ 
is a pair of non-negative 
 integers $p,q$ such that
 $p\le q$ and 
 the Betti numbers $\beta_j(X)$ are trivial for $p < j < q$. Denote the
 gap by the closed interval $[p,q]$.  
 \end{defn}
 
 \begin{rem} For intervals of the form $[p,p+1]$, the gap condition is automatically satisfied.
 \end{rem}

 Let $\Bbb F$ be a field of characteristic zero. In what follows, chain complexes and their homology
will always be taken with coefficients in $\Bbb F$, even though
almost always the coefficients are suppressed from the notation.

Given a gap $[p,q]$ and a good protocol $\gamma:\Sigma \to \breve{\cal M}_{p,q}$, 
we will construct a linear transformation of vector spaces
 \begin{equation}\label{eqn:J}
\cal J_{p,q}\: H_{q-p}(\Sigma) \otimes H_{p}(X) \to H_{q}(X)
\end{equation}
called the {\it hypercurrent homomorphism} (we reiterate that
homology is taken with coefficients in  $\Bbb F$). The following result enumerates 
some of its basic properties.

\begin{bigthm} \label{bigthm:properties} The hypercurrent homomorphism \eqref{eqn:J}  
exhibits the following properties:
\begin{enumerate}
\item (Functoriality). $\cal J_{p,q}$ is contravariantly functorial in $\Sigma$;
\item (Homotopy Invariance). $\cal J_{p,q}$ is homotopy invariant: if there is
 homotopy from $\gamma$ to $\gamma'$ in the space of good protocols, then
the associated hypercurrent homomorphisms coincide.
\item (Initial Condition). When $q=p$, the homomorphism 
\[
\cal J_{p,p}\: H_{0}(\Sigma) \otimes H_{p}(X) \to H_{p}(X)
\]
is the identity. 
\item (Non-triviality). When $q >p$,  there exists a finite connected CW complex
$X$ with gap $[p,q]$ and a good protocol $\Sigma \to \breve{\cal M}_{p,q}(X)$ such that $\cal J_{p,q}$ is nontrivial.
\end{enumerate} 
\end{bigthm}

\begin{rems} (1). The hypercurrent homomorphism is also natural
in the field $\Bbb F$ by extension of scalars. In particular, it suffices
to consider the case when $\Bbb F = \Bbb Q$ is the rational numbers.
\smallskip

\noindent (2). When $\dim X=1$,
the homomorphism  
\[ 
\cal J_{0,1}\: 
H_{1}(\Sigma) \otimes H_{0}(X) \to H_{1}(X)
\]
coincides with the average current of \cite{CKS13} evaluated at
the generator of $H_0(X)$ given by any zero cell.
\smallskip

\noindent (3). When $\dim X = d > 1$, the homomorphism $\cal J_{d-1,d}$
coincides with the average current homomorphism of \cite{CCK16} evaluated
at the homology class of the higher Boltzmann cycle $[\rho^B]\in H_p(X)$
(cf. \cite[defn.~1.12]{CCK15b}).
\smallskip

 \noindent (4). Set $H^p_q(X) := H^p(X;H_q(X))$.
 Since $\cal J_{p,q}$ is functorial in $\Sigma$,
we  tend to view $\cal J_{p,q}$, or equivalently, its adjoint
\[
{\cal J}^\ast_{p,q} \in H^{q-p}(\Sigma; H^{p}_q(X)) \, ,
\]
 as a kind of {\it characteristic
class} for spaces $\Sigma$ equipped with a good protocol: the operation $\Sigma \mapsto H^{q-p}(\Sigma; H^{p}_q(X))$
defines a presheaf on the category of good protocols
and the hypercurrent homomorphism defines a global section of this presheaf.
\end{rems}

The following  hints that the hypercurrent homomorphism
measures in some way the difference between the non-triviality of the boundary operator
in the cellular chain complex of $X$ and the triviality of the homology in the gap degrees.

\begin{bigadd} Assume $X$ has a gap $[p,q]$.  If
there is a $j\in [p,q{-}1]$ such that 
\begin{itemize} 
\item the cellular boundary operator 
$\partial\: C_{j+1}(X) \to C_{j}(X)$ 
is trivial, or
\item there is a $j\in [p,q]$ such that $|X_j|\le 1$,
\end{itemize}
then the hypercurrent homomorphism $\cal J_{p,q}$ is trivial.
\end{bigadd}

\subsection{The hypercurrent chain map \label{subsec:hypercurrent-chain}}
Let $C_\ast(X)$ be the cellular chain complex over $\Bbb F$ of the CW complex $X$.
Then $C_j(X)$ is the vector space over $\Bbb F$ with basis $X_j$. 
Let $\bar C_\ast = \bar C_\ast(X)$ be the chain complex  given by
\[
\bar C_j := C_{j+p}(X^{(q)},X^{(p-1)})\, ,
\]
i.e., the cellular chain complex of the pair $(X^{(q)},X^{(p-1)})$ shifted
by $p$.
There is an evident surjection of vector spaces
\[
\hom(H_0(\bar C_\ast),H_{q-p}(\bar C_\ast)) \to  \hom(H_p(X),H_q(X))
\]
(cf.~Remark \ref{rem:proj-rest-surj} below).
The homological hypercurrent homomorphism will be induced by a
chain map 
\begin{equation} \label{eqn:hyper-map-exist}
\cal J\: I_\ast(\Sigma) \otimes \bar C_\ast \to \bar C_\ast \, ,
\end{equation}
where $I_\ast (\Sigma)$ is the chain complex over $\Bbb F$ freely generated by 
the set of ``small'' singular simplices $\sigma\:\Delta^k \to \Sigma$. Here ``small'' means that
there is a $j\in [p,q]$ such that 
$W_j(b) \:X_j \to \Bbb R$ is one-to-one for all  $b\in \sigma(\Delta^k)$. 
A subdivision argument shows that inclusion of $I_\ast (\Sigma)$
into the full total singular complex $S_\ast(\Sigma)$ is a quasi-isomorphism.

\begin{bigthm} \label{bigthm:hypercurrent-chain} The hypercurrent chain map \eqref{eqn:hyper-map-exist} exists and is well-defined up to contractible choice.
\end{bigthm}

Theorem \ref{bigthm:hypercurrent-chain} is proved using Quillen model
category machinery applied to a certain functor category.
With the intent of clarifying the way in which the current 
work relates to \cite{CKS13}, 
we discuss two cases: (i) dimension one and (ii) the general case.

\subsection{The graph case}  
Consider the case $[p,q] = [0,1]$, where $X = (X_0,X_1)$ is a graph. 
In this instance, the chain map we seek
arises from a map of spaces. 
We will sketch below a form of the construction using the methods of \cite{CKS13}.

If $X$ is a graph, then 
\[
\Sigma = \Sigma(0) \cup \Sigma(1)\, ,
\] 
where
 $W_j(b)\: X_j \to \Bbb R$ is one-to-one for $b\in \Sigma(j)$.
 If $b\in \Sigma(0)$, then $W_0(b)\: X_0 \to \Bbb R$ has a 
 unique minimum $L_b \in X_0$, which is a zero cell of $X$.
 If $b\in \Sigma(1)$, then the greedy algorithm applied to
$W_1(b)\: X_1 \to \Bbb R$ produces a minimal spanning tree $T_b$.
Thus to every point of $\Sigma$ the assignment 
\[
b\mapsto 
\begin{cases}
L_b \quad & \text{ if } b\in \Sigma(0); \\
T_b     &\text{ otherwise.}
\end{cases}
\]
codifies a quasifibration $N \to \Sigma$ with contractible fibers, where
$N$ is the space of pairs $(b,x)$ with $b\in \Sigma$ and where $x$ is either $L_b$
or a point of $T_b$.
It follows that the map $N \to \Sigma$ is a weak homotopy equivalence. 
Second factor projection defines a map $N \to X$. Consequently, modulo technical details, we have defined a weak map of spaces
\[
\Sigma @< {}_\sim << N @>>> X
\]
which induces a space level version of the desired chain map. 

\subsection{The general case} 
When $q-p> 1$, the  above approach
doesn't generalize. 
Let $I_{\Sigma}$ be the poset of small singular simplices partially ordered with respect to facial inclusion.  
The map \eqref{eqn:hyper-map-exist} is defined by a local construction in sense that
it is the induced map of  homotopy colimits
associated with functors appearing in a canonical chain of natural transformations of chain complex valued functors $I_{\Sigma} \to \Ch$. The chain of natural 
transformations has the form
\begin{equation} \label{eqn:natural-pre-hyper}
 \chi @>>> H_\ast(\tau) @<{}_\sim << \tau @>>> \chi\, ,
\end{equation}
in which 
\begin{itemize} 
\item the functor $\chi$ is the constant functor with value $\bar C_\ast$;
\item the functor $\tau$, called the {\it tree functor}, is
defined in terms of higher dimensional spanning tree and co-tree data in the CW complex $X$ (cf.~\S\ref{sec:spanning-tree} and Definition \ref{defn:tree-functor});
\item the functor $H_\ast(\tau)$ is given by taking the homology of $\tau$ objectwise, where the homology of a chain complex is considered as a chain complex with trivial boundary operator.
\item The natural transformation $\tau \to \chi$ is induced by objectwise inclusion, and the natural transformation $\tau \to H_\ast(\tau)$ exists as the functor $\tau$ is objectwise acyclic in positive degrees. 
\item The natural transformation $\chi \to H_\ast(\tau)$ is  determined by 
specifying its value in degree 0, the latter which we take to be the projection
onto 0-dimensional homology.
\end{itemize}
The above leads to a description of $\cal J$ as arising from a map in the homotopy category of functors. To obtain an actual map in the functor category itself, one has to work a bit harder, appealing to model category machinery.

\subsection{Examples} For $q \ge 1$, let $X^q = S^q$ be the sphere of dimension $q$ equipped with the CW composition
in which there are two $j$-cells $e^j_\pm$ in each dimension $j \le q$  given by the upper and lower hemispheres
of $S^j \subset S^q$. 
Let $\Sigma := \breve{\cal M}_{0,q}(X^q)$ be the (universal) good protocol 
for $X^q$.
By Proposition \ref{prop:good-type} below, there is a
homotopy equivalence $\Sigma \simeq S^q$).

Let $Y^q$ be the CW complex obtained from $X^{q-1}$ 
by attaching two additional $q$-cells. One of the $q$-cells is attached using the identity map
and the other is attached using the constant map to a 0-cell. Then $\breve{\cal M}(Y^q) = \breve{\cal M}(X^q)$.
In particular, $\Sigma$ is also a good protocol for $Y^q$.
Furthermore, $Y^q$ is homotopy equivalent to $X^q$ (in fact,  $Y^q$ is homeomorphic to the wedge $D^q \vee S^q$).
Consequently, $X^q$ and $Y^q$ are a pair of homotopy equivalent CW complexes having the same number of cells in each dimension. The following result shows that $\cal J$ distinguishes
 the cell structures.

\begin{bigthm}\label{bigthm:examples}  (a). The hypercurrent homomorphism
\[
\cal J_{0,q}: H_q(\Sigma) \otimes H_0(X^q) \to H_q(X^q)
\]
is an isomorphism of vector spaces.

\noindent (b). The hypercurrent homomorphism
\[
\cal J_{0,q}\: H_q(\Sigma) \otimes H_0(Y^q) \to H_q(Y^q)
\]
is trivial.
\end{bigthm}

\begin{rem} The case $q=1$ was discussed in \cite[ex.~7.9]{CKS13}.
\end{rem}

\subsection{Analytical hypercurrents} \label{subsec:analytical}
In what follows, we work over the field $\Bbb F = \Bbb R$ of real numbers.
The space 
$\Sigma$ will be
a compact smooth manifold, possibly with boundary and $X$ will be a
finite connected CW complex with gap $[p,q]$. 

We will describe a map
\begin{align} \label{eqn:map-analytic}
\cal J^{\text{an}}  \: \cal M_{p,q}^{\Sigma} & \to \Omega^\ast(\Sigma;\End(\bar C))\, , \\
\gamma & \mapsto \cal J^\an(\gamma)\notag
\end{align}
where the source of \eqref{eqn:map-analytic} 
is the the function space $C^{\infty}(\Sigma,\cal M_{p,q})$, i.e.,
the space of smooth protocols $\gamma\:\Sigma \to 
 \cal M_{p,q}$ (not necessarily good).  
The target of \eqref{eqn:map-analytic} is the total
complex of the  de~Rham complex of $\Sigma$ with coefficients in the 
chain complex 
\begin{equation} \label{eqn:endo-complex}
(\End(\bar C),\eth)
\end{equation}
of endomorphisms of $\bar C$, with boundary operator $\eth$ (cf.~\S\ref{sec:prelim}).

 The map \eqref{eqn:map-analytic}
 is natural with respect to smooth maps $\Sigma \to \Sigma'$. 
The form $\cal J^\an(\gamma)$ is of total degree zero. We will characterize
it by three axioms. The characterization requires some preparation.

\begin{notation} 
When $\gamma \in\cal M_{p,q}^{\Sigma}$ is understood and there is no ambiguity, 
we drop the argument and write $\cal J^\an$ in place of
$\cal J^\an(\gamma)$ to avoid notational clutter.
\end{notation}

\begin{defn}[De~Rham Complex] If $V$ is a finite dimensional real vector space, let
\[
\Omega^\ast(\Sigma;V) = \Omega^\ast(\Sigma)\otimes_{\Bbb R} V\, ,
\]
where  $\Omega^\ast(\Sigma)$ is the 
de~Rham complex of smooth differential forms
on $\Sigma$ 
Note that the operator $d\otimes \text{id}_V$,
makes $\Omega^\ast(\Sigma;V)$ into a cochain complex,
where $d$ is the usual exterior derivative on differential forms.  

Note that
$\Omega^k(\Sigma;V)$ is identified with the space of smooth sections of the 
bundle over $\Sigma$ whose fiber at
$b\in \Sigma$ is the space of linear maps 
\[
\Lambda^k T_b\Sigma \to
V\, ,
\] where $\Lambda^k T_b\Sigma$ is the $k$-th exterior power of the tangent space
to $\Sigma$ at the point $b$.

If $V_\ast$ is a graded vector space, then we define $\Omega^\ast(\Sigma;V_\ast)$ to be
the graded cochain complex $\oplus_j \Omega^\ast(\Sigma;V_j)$. If $(V_\ast,\partial_V) $ is a chain complex, then $\Omega^\ast(\Sigma;V_\ast)$ has the structure of a bi-complex, and the 
 associated total complex has boundary operator $d + (-1)^\ast\partial_V$. 
\end{defn}

\begin{notation} If $\phi\: V_\ast \to V_\ast$ is a graded map, then
it induces in the evident way a map $\Omega^\ast(\Sigma;V_\ast) \to \Omega^\ast(\Sigma;V_\ast)$, which we also denote by $\phi$.
\end{notation}

\begin{defn}[Modified Inner Product]
For $W_\bullet\in \cal M_{p,q}$, the {\it modified inner product} on $\bar C_k$ is defined
on basis elements $x,y\in X_{k+p}$ by
\[
\langle x,y\rangle_W := e^{W_{k+p}(x)}\delta_{xy}
\]
where $\delta_{xy}$ is the Kronecker delta of $x$ and $y$.  
Similarly, if $\gamma\: \Sigma \to \cal M_{p,q}$ is a protocol, 
then we obtain a family of  modified inner products on $\bar C_k$ parametrized by $\Sigma$.
\end{defn}

\begin{defn} Let 
\[
\bar D_{W\ast} \subset \bar C_\ast
\]
be the graded vector subspace defined as follows:
\begin{itemize}
\item If $j \notin  [0,q-p]$, then
we take $\bar D_{Wj}$ to be trivial.
\item If $j \in (0,q-p]$, then
$\bar D_{Wj}$ is defined to be the orthogonal complement to the subspace of $j$-cycles $\bar Z_j$ in the modified inner product.  
\item If $j = 0$, then $\bar D_{W0}$  is defined to be the orthogonal complement to the 
subspace of $0$-boundaries $\bar B_0$ in the modified inner product.
\end{itemize} 
\end{defn}

We return to the problem of characterizing $\cal J^\an$.
By the dimensional constraints on $\bar C$, the form is a finite direct sum
\[
\cal J^\an  = \sum_{\ell= 0}^{q-p} \cal J^\an_\ell\, ,
\]
in which $\cal J^\an_\ell \in  \Omega^{\ell}(\Sigma;\End(\bar C)_\ell)$.

The forms $\cal J^\an_{\ell}$ are subject to the following axioms:
\begin{itemize}
\item[(A1)] (Continuity Equation). For $\ell \ge 1$, the  identity
\[
\eth \cal J^\an_\ell = d \cal J^\an_{\ell-1}
\]
is satisfied, where $\eth$ is given by \eqref{eqn:endo-complex}.
\item[(A2)] (Orthogonality Condition). The form
\[
\cal J^\an_\ell \in \Omega^\ell(\Sigma; \End(\bar C)_\ell)
\]
lies in the subspace $\Omega^\ell(\Sigma; \hom(\bar C,\bar D_W)_\ell)$.
\end{itemize}
Axiom A1 implies that the image of 
$\cal J^\an_{0}\: \Sigma \to \End(\bar C)_0$
lies in subspace of
chain maps $\bar C \to \bar C$. 
By taking $0$-th homology, we infer that
$\cal J^\an_{0}$ induces
 a $0$-form
 \begin{equation} \label{eqn:homology-induced}
\Sigma \to \hom(H_0(\bar C), H_0(\bar C))\, .
\end{equation}
\begin{itemize}
\item[(A3)] (Initial Value).  The form \eqref{eqn:homology-induced} is constant with value the identity map.
\end{itemize}

\begin{bigthm} \label{bigthm:uniqueness} 
There exists precisely one map \eqref{eqn:map-analytic}
satisfying axioms A1-A3.
\end{bigthm}

\subsection{Quantization} Our conventions for our last main result are as follows:
$\Bbb F = \Bbb R$ will be the real numbers and
$\Sigma$ will be a smooth manifold. The protocol $\gamma\:\Sigma \to \cal M_{p,q}$ 
will be a smooth map. We let $I_\ast(\Sigma)$ denote the chain complex of small
smooth singular simplices in $\Sigma$.
For any real number $\beta >0$, 
one has an associated protocol $\beta\gamma$ defined by pointwise scalar multiplication in $\cal M_{p,q}$.\footnote{The parameter $\beta$ is to be regarded as inverse temperature.}

Let $R_\ast \subset I_\ast(\Sigma)$ be any chain subcomplex generated by small smooth singular simplices of dimension $\le q-p$.
We define 
\begin{equation} \label{eq:anal-top-comparison}
\cal J^\an(\beta\gamma)_\sharp \: R_\ast \to \End(\bar C)
\end{equation}
by the formula
\[
\cal J^\an(\beta\gamma)_\sharp( \Delta^j @>\sigma >> \Sigma) \,\, :=\,\,  \textstyle \int_{\Delta^j} \sigma^\ast(\cal J^\an(\beta\gamma))\, .
\]
In \S\ref{section:analytical}
we will show that $\cal J^\an(\beta\gamma)_\sharp$ is a chain map.
The following result pins down the relationship between the analytical and topological hypercurrent maps (compare \cite[thm.~A]{CKS13} and \cite[thm.~C]{CCK15b}).

\begin{bigthm}[Quantization] \label{bigthm:quantize} In
the low temperature limit $\beta\to \infty$, the chain homotopy class of $\cal J^\an(\beta\gamma)_\sharp$
converges to the chain homotopy class of the restriction to $R_\ast$ of the hypercurrent chain map $\cal J$.
\end{bigthm}

\begin{rem} For a more detailed statement, see Theorem \ref{thm:quantization}.
Theorem \ref{bigthm:quantize} is to be viewed as a  `fractional' quantization result because
$\cal J$ is defined over the rational numbers (and extended by scalars to the reals) 
whereas the analytical current is not. It can also be thought of as an index
theorem relating an analytically defined invariant to a topologically defined one.
\end{rem}

\subsection{Summary}
Suppose that $\gamma\: \Sigma\to \cal M_{p,q}(X)$ is a smooth protocol, with $\Sigma$ a closed Riemannian manifold of
dimension $n := q-p$. The paper \cite{CKS13} deals with the case $p=0$ and $n =1$. We showed there that $\gamma$  determines a continuous time Markov chain whose state diagram is the double of the $1$-skeleton of $X$. The evolution of the system in this case is governed by the Kolmogorov equation which is a first order ODE acting on distributions (0-chains).  
The average current was a homology clas associated with the {\it adiabatic limit} of the system (the adiabatic limit is given by rescaling the metric on $\Sigma$ 
so that the total length tends to $\infty$). In \cite{CCK16} we
generalized \cite{CKS13} to $p > 0$ with $q=p+1$. In this case the double of the 1-skeleton of $X$ is replaced by
the {\it cycle incidence graph} of the $(p+1)$-skeleton of $X$. The latter is the (possibly infinite) graph whose vertices are the cellular $p$-cycles of $X$ in a given homology class in which an edge is determined by an elementary homology 
between cycles, where an elementary homology is given by taking a suitable scalar multiple of a  $(p+1)$-cell.
In this case the process is a continuous time biased random walk on the cycle incidence graph. To avoid the technical problems
of working with the Kolmogorov equation acting on infinite dimensional Hilbert spaces, 
we worked with a  {\it dynamical equation} which acts on the finite dimensional vector space of cellular $p$-cycles in $X$.
By contrast, when $q-p > 1$ we do not know of any physical or dynamical system which gives rise to the hypercurrent. When $p=0$ and $q-p> 1$, then we suspect that the hypercurrent is an invariant associated with an as yet to be defined {\it statistical field theory} based on $(q-p)$-branes. The above discussion is summarised in the following table:

{\tiny
\begin{table}[ht]
\centering 
\begin{tabular}{c c c c c c} 
\hline\hline 
$p$ & $q-p$ & Process & States & State Diagram & Invariant \\ [0.5ex] 
\hline 
 0 & 1 & biased random walk & vertices of $X$ & double of $X^{(1)} $ & average current \\ 
$>  0$ & 1 & biased random walk  & $p$-cycles of $X$ & cycle-incidence graph of $X^{(q-p)}$ & average current  \\
 0 & $> 1$ &  $(q-p)$-brane theory? & ? & ? & hypercurrent   \\
\hline 
\end{tabular}
\label{table:nonlin} 
\end{table}
}

\begin{out} In Section \ref{sec:prelim} we review the standard properties of chain complexes over a field of characteristic zero. 
Section \ref{sec:functor-cat-chain} introduces the projective model structure on the  category of chain complex valued functors from a suitable small category. We also prove a kind of acyclic models result which will enable us to construct the hypercurrent map. Section \ref{sec:spanning-tree} reviews spanning trees and spanning co-trees in higher dimensions;  this material is lifted from \cite{CCK15a} and \cite{CCK15b}. In section \ref{sec:hyper-construct} we construct the hypercurrent chain map. In section
\ref{sec:good-weights}, we show that the space of good weights has the homotopy type of a wedge of spheres. Section \ref{sec:proof-props-ex} contains proofs of Theorems \ref{bigthm:properties} and  \ref{bigthm:examples}. In section
\ref{sec:robust} we determine the homotopy type of the space of robust weights and extend  the hypercurrent homomorphism to robust protocols.  In Section \ref{section:analytical} we define and axiomatically characterize the analytical hypercurrent.
In section \ref{sec:quantize} we provide a proof of the Quantization Theorem. Lastly, in section \ref{sec:space-level} we construct space level hypercurrent maps in certain cases. 
\end{out}

\begin{acks} The third author was partially supported by Simons Foundation Collaboration
Grant 317496.
\end{acks}

 \section{Preliminaries}\label{sec:prelim}
\subsection{Chain complexes}
Let 
$\Bbb F$
 be a field of characteristic zero. 
Let $C$ be a chain complex over $\Bbb F$ with boundary operator
$\partial$.
In what follows all chain complexes are unbounded in the sense
that they indexed over $\Bbb Z$.
Recall that a quasi-isomorphism  $C\to D$ is a chain map which induces
an isomorphism on homology. A morphism $C\to D$ is a quasi-isomorphism if and only if it
is a chain homotopy equivalence, since we are working over a field. 
Additionally for any chain complex $C$, there is a quasi-isomorphism
\[
(C,\partial) \to (H_\ast(C),0)
\]
which induces the identity map in homology, where the target is the
homology of $C$ with trivial boundary operator. The quasi-isomorphism is constructed as follows:
equip $C$ with an inner product in each degree and decompose $C_k$ orthogonally as
\[
B_k \oplus {\cal H}_k \oplus B_k^\ast
\]
where $B_k$ is the vector space of $k$-boundaries and ${\cal H}_k$ is the orthogonal complement of $B_k$ in the vector space of $k$-cycles $Z_k$.  With respect to this
decomposition, the boundary operator is trivial on the first two summands and
gives an isomorphism $B_k^\ast @> \cong >> B_{k-1}$. The desired quasi-isomorphism is then defined by sending a vector $(x\oplus y \oplus z)$ to $[x] \in H_k(C)$, where $[x]$
denotes the homology class of the cycle $x\in  {\cal H}_k$.

Suppose $f\: C\to D$ is a chain map. If $f$ is null homotopic, then $f$ induces the trivial map in homology. The converse is also true as is easily seen by replacing $C$ and $D$ with the chain map $f_\ast \: (H_\ast(C),0)\to (H_\ast(D),0)$.

Let $\hom(C,D)$ be the internal hom-complex in which
\[
\hom(C,D)_n = \prod_{i\in \Bbb Z} \hom(C_i,D_{n+i})\, ,
\]
with boundary operator
\[
\eth(f) = \partial_D f - (-1)^{n} f \partial_C\, ,  \qquad f\in \hom(C,D)_n\, .
\]  

\begin{notation} When $C = D$ we write
\[
\End(C) = \hom(C,C)\, .
\]
\end{notation}

\subsection{Model structure}
Let $\Ch$ denote the category of unbounded chain complexes over 
$\Bbb F$ equipped with the projective model structure \cite[\S2.3]{Hovey}.
A weak equivalence in $ \Ch$
is a quasi-isomorphism. 
A morphism is a fibration if and only if it is surjective in every degree.
Every object of $\Ch$ is fibrant.

 A morphism is a cofibration if it has the left lifting property with respect to the trivial (acyclic) fibrations. An equivalent characterization of cofibrations is given by
attaching cells: for $j \in \Bbb Z$, we let $D^j_\ast$ be the chain complex
which is given by $\Bbb F$ in degrees $j$ and $j-1$ and trivial otherwise,
where the boundary operator in degree $j$ is given by the identity map. Let
$S^{j-1}_\ast$ be the chain complex which is $k$ in degree $j-1$
and trivial otherwise;  one has an inclusion $S^{j-1}_\ast \to D^j_\ast$. Then
the cofibrations of $\Ch$ are generated by these maps in the following sense:
if $C_\ast$ is a chain complex and $f\: S^{j-1}_\ast \to C_\ast$ is a chain map, one can form the algebraic mapping cone
\[
M(f)_\ast := C_\ast \oplus_f D^j_\ast
\]
which is defined as the pushout of the diagram $C_\ast @< f << S^{j-1}_\ast \subset D^j_\ast$. 
With respect to the inclusion $C_\ast \to M(f)_\ast$, 
one says  that $M(f)_\ast$ is obtained from $C_\ast$ by attaching a $j$-cell. 
A cofibration of $\Ch$ is 
a  map $C_\ast \to D_\ast$ which is the result of iteratively attaching (possibly transfinitely many) cells, or is a retract thereof.

 Every bounded
below object is cofibrant \cite[lem.~2.3.6]{Hovey}. A morphism is a cofibration
if and only if it is degreewise injective with cofibrant cokernel 
\cite[prop.~2.3.9]{Hovey}.
In particular, a degreewise injective morphism of bounded below objects is a 
cofibration.

\subsection{Truncation/shift\label{subsec:trunc-shift}}
Recall from the introduction that $\bar C_\ast$ is the chain complex
given by $\bar C_j = C_{j+p}(X^{(q)},X^{(p-1)})$, in other words
\[
\bar C_\ast = C_\ast((X^{(q)},X^{(p-1)})[p]
\]
is the $p$-fold desuspension of the cellular chain complex 
of the CW pair $(X^{(q)},X^{(p-1)})$.

\begin{rem} \label{rem:proj-rest-surj}
For any $X$, there are evident homomorphisms 
\[
H_p(X) \to H_0(\bar C) \,\, 
\text{ and } \,\, H_{q-p}(\bar C)  \to H_q(X)\, ,
\] 
where the former is injective and the latter
is surjective. 
In particular, 
projection/restriction defines a surjective homomorphism of vector spaces
\[
\hom(H_0(\bar C),H_{q-p}(\bar C)) 
\to \hom(H_p(X),H_q(X)) \, .
\]
This last map shows $\bar C_\ast$ can be used in place of 
the cellular chain complex $C_\ast(X)$ when attempting to define the 
hypercurrent homomorphism $\cal J_{p,q}$.
\end{rem}

\section{Functor categories of chain complexes \label{sec:functor-cat-chain}}

Let $I$ be a partially ordered set such that any element $x\in I$ 
has finite degree in the sense that the longest chain of strictly increasing objects
less than $x$ is finite. Call the length of this chain $\deg(x)$. The function
\[
\deg\: I \to \Bbb N
\]
is a morphism of partially ordered sets, i.e., it is a functor.

\begin{ex}\label{ex:simplex-cat} Let $S_B$ denote the poset of singular simplices in a space $B$.
An element of $S_B$ is just a singular simplex $\Delta^j\to B$ and the partial ordering
is defined by facial inclusion. Then every element of $S_B$ has finite degree.
\end{ex} 

As in \S\ref{sec:prelim}, $\Ch$ will denote the category of unbounded chain complexes over $\Bbb F$.
By \cite[th.~5.1.3]{Hovey}, the functor category 
\[\Ch^I\, , \]
forms a model category
in which the weak equivalences/fibrations are the objectwise equivalences/fibrations  and the cofibrations are defined to be those
maps satisfying the left lifting property with respect to the acyclic fibrations.

We now give an explicit characterization of the cofibrations. Suppose
If $X\: I \to  \Ch$ is a functor and $r\in I$ is an object. Define the 
{\it latching object}
\[
L_rX := \colim_{s < r}X(s)\, .
\]
The assignment $X \mapsto L_rX$ defines a functor $L_r \: I \to  \Ch$, and
 one has a natural latching map $L_r(X) \to X(r)$.  Then $X$ is cofibrant if and only if the latching map is a cofibration of $ \Ch$ for all objects $r$. 
More generally, a morphism $A \to X$ of $\Ch^I$ is a cofibration if and only if
for every $r\in I$ the relative latching map 
\[
L_r X\oplus_{L_r A} A(r) \to X(r)
\]
is a cofibration of $\Ch$, where the domain of this map is 
the pushout of $L_r X @<<< L_r A @>>> A(r)$.

\begin{ex} 
Fix $E\to B$, a Serre fibration of spaces. 
 Define a functor
\[
F\: S_B \to \Ch
\]
by the rule 
\[
F(\sigma\: \Delta^j \to B) =  S_\ast(\sigma^\ast E)\, ,
\] 
where $\sigma^\ast E$ denotes the pullback of $E \to B$ along $\sigma$
and $S_\ast(\sigma^\ast E)$ is its singular chain complex over $\Bbb F$.  
Then the functor $S_\ast$ is cofibrant.
\end{ex}
 
Recall that a chain complex $C_\ast$ is bounded below by 0 if $C_k = 0$ for $k < 0$.

\begin{defn} A functor $F\: I \to \Ch$ is  {\it bounded below by 0}
if the chain complex $F(r)$ is is  bounded below by 0 for all $r\in I$.

A functor $G\: I \to  \Ch$ 
is  {\it positively acyclic} if 
\begin{itemize}
\item $G$ is bounded below by 0, and
\item for all objects $r\in I$, the homology groups
 $H_\ast G(r)$ are trivial in degrees $\ast \ne  0$.
 \end{itemize}
Recall that {\it cofibrant approximation} of an object $F\in \Ch^I$
consists of a cofibrant object $F^c$ equipped with a weak equivalence $F^c @> {}_\sim >> F$.
\end{defn}

\begin{ex} If $G$ is positively acyclic, then the evident natural
transformation 
\[
G \to H_\ast G\, ,
\]
is an acyclic fibration (i.e., both a fibration and a weak equivalence).
\end{ex}

The next result is similar in spirit to the classical acyclic model theorem, although the technical assumptions are somewhat different.

\begin{prop}[``Acyclic Models''] \label{prop:lift}  
Let $F,G\: I \to \Ch$ be functors where $F$ is bounded below by 0 and
$G$ is positively acyclic. Let $F^c @>{}_\sim>>  F$ be a cofibrant approximation. Then 
any natural transformation $\alpha_\ast \:H_\ast F \to H_\ast G$ is induced by a natural transformation 
\[
F^c\to G\, ,
\] 
In particular, when $F$ is cofibrant
there is a natural transformation $F \to G$ that induces $\alpha_\ast$.
\end{prop}

\begin{proof} The function $\deg$ equips $I$ with the structure of a direct
category in the sense of \cite[defn.~5.1.1]{Hovey}. Let $F^c_0$ be the degree zero part of $F^c$,
and let $u\: F^c_0 \to H_0F$
be the evident map, . 
Consider the a lifting problem
\begin{equation} \label{eqn:lift-diagram}
\xymatrix{
& G \ar[d] \\
F^c \ar[r]_(.4){\alpha} \ar@{..>}[ur]
& H_\ast G\, ,
}
\end{equation}
where $\alpha\:F^c\to H_\ast G$ is the natural transformation defined
by
\[
\begin{cases} 
\alpha_\ast\circ u \quad & \text{in degree 0}\\
0  &\text{otherwise.}
\end{cases}
\]
The existence of a lift follows directly from the model category lifting axiom since $F^c$
is cofibrant and $G \to H_\ast G$ is an acyclic fibration.
\end{proof}

\begin{cor}[Rigidity] \label{cor:lift} Assume $F$, $G$ and $\alpha_\ast$ satisfy the conditions
of Proposition \ref{prop:lift}. Then 
a natural transformation 
$F^c\to G$ inducing $\alpha_\ast$ exists and  is unique up to contractible choice.
\end{cor}

 \begin{proof} Once basepoint $F^c \to G$ in the affine space of lifts has been chosen,
the set of lifts is in bijection with the set of lifts of the trivial map. The latter has the structure of the vector space of zero chains of the kernel of the map 
of complexes $\hom(F^c,G) \to \hom(F^c,H_\ast G)$. But this kernel is contractible because 
the natural transformation $G\to H_\ast G$ is both a weak equivalence and a fibration.
\end{proof}

\section{Spanning trees and co-trees \label{sec:spanning-tree}}

Let $X$ be a finite connected CW complex. 
In the following definitions, $\beta_j(X) = \dim_{\Bbb F} H_k(X)$ denotes the 
$j$th Betti number of $X$, and $X^{(j)}$ denotes the $j$-skeleton of $X$.

\begin{defn}[{\cite[defn.~1.2]{CCK15a}}] \label{defn:spanning-tree} A subcomplex $T\subset X$ is a {\it spanning tree}  
in degree $d$ if
\begin{itemize}
\item $H_d(T) = 0$,
\item $\beta_{d-1}(T) = \beta_{d-1}(X)$, and
\item $X^{(d-1)} \subset T\subset X^{(d)}$.
\end{itemize}
\end{defn}

\begin{ex} If $d=1$, then
the above coincides with the usual definition of a spanning tree for the $1$-skeleton of $X$, considered as a graph.
\end{ex}

\begin{defn}[{\cite[defn.~1.9]{CCK15b}}] \label{defn:spanning-cotree} A 
{\it spanning co-tree} in degree $d$ for $X$ is a 
  subcomplex $L \subset X$ such that
  \begin{enumerate}
  \item The inclusion $L \subset X$ induces an isomorphism of vector spaces
  \[
  H_{d}(L) @> \cong >> H_{d}(X)\, ;
  \]
  \item $\beta_{d-1}(L) = \beta_{d-1}(X)$;
  \item $X^{(d-1)} \subset L \subset X^{(d)}$.
  \end{enumerate}
\end{defn}

\begin{ex} A spanning co-tree in degree 0 is just a 0-cell of $X$.
\end{ex}

\begin{ex} If $H_d(X) = 0$ then a spanning tree in degree $d$ is the same thing
as a spanning co-tree in degree $d$.
\end{ex}

\begin{lem}[{\cite[lem.~2.3,2.4]{CCK15a}}, {\cite[lem.~2.2,2.3]{CCK15b}}] Let
$X$ be a connected finite CW complex. Then
$X$ has a spanning tree and spanning co-tree in every degree $d \ge 0$
and in each case there are only finitely many.
\end{lem}

\begin{defn} \label{defn:d-tree} Let $[p,q]$ be a gap. If $p \le d \le q$ then a {\it $d$-tree} for $X$ is subcomplex
$T\subset X$ such that
\begin{itemize}
\item if $d = p$, then $T$ is a spanning co-tree;
\item if $d > p$, then $T$ is a spanning tree.
\end{itemize}
\end{defn}

\section{Chain level hypercurrents \label{sec:hyper-construct}}

Assume $\gamma\: \Sigma \to \cal M_{p,q}$ is a good protocol.
Recall that 
\[
\gamma(b) = (W_p(b),\dots, W_q(b))\, \]
 where $W_j(b) \: X_j \to \Bbb R$.

\begin{defn}[Stratification of $\Sigma$]\label{defn:stratification}
Let $\Sigma_j$ be the set of $b\in B$ such that the function 
$W_j(b)$ is one-to-one.
Then $\Sigma_j \subset \Sigma$ is open and $\Sigma = \amalg_j \Sigma_j$ as sets.
 \end{defn}

\begin{defn}[Small simplices] \label{defn:small}
A singular simplex
$\sigma\: \Delta^j \to \Sigma$
is {\it small} if its image is contained in $\Sigma_j$ for some index $j$.
\end{defn}

Let \[
I_{\Sigma}
\] be the poset of small singular simplices in $\Sigma$. An object
of $I_\Sigma$ is a small singular simplex and a morphism is an inclusion of
faces. The $I_{\Sigma}$ is a poset in which every element has finite degree.

\begin{notation}
If $\sigma\: \Delta^j \to \Sigma$ is a (small) singular simplex we set 
\[
j_\sigma := j\, .
\]
\end{notation}

\begin{defn}[Tree functor] \label{defn:tree-functor} 
For a small simplex $\sigma\: \Delta^j \to \Sigma$, let 
\[
k_\sigma \in \Bbb N
\]
be the smallest integer such that $\sigma(\Delta^j) \subset \Sigma_{k_\sigma}$. 
By the greedy algorithm (cf.\ \cite[p.~9]{CCK15b}) there is a preferred 
$k_\sigma$-tree $T_\sigma$ associated with $s$.

The assignment $\sigma\mapsto \bar C(T_\sigma)$ defines a functor  
\[
\tau\: I_{\Sigma} \to \Ch\, .
\]
Call this the {\it tree functor}. (We remind the reader that $\bar C(T_\sigma)$ is the chain complex
with $\bar C_j(T_\sigma) = C_{p+j}(T^{(q)}_\sigma,T^{(p-1)}_\sigma)$.)
\end{defn}

\subsection{The pre-hypercurrent map}
Let $\Delta_\ast^j$ denote the simplicial chain complex of the standard $j$-simplex.
Consider the following three functors $I_{\Sigma} \to \Ch$: 
\begin{enumerate}
\item The constant functor $\chi$ given by
 \[
\sigma\mapsto  \bar C :=\bar C(X)\, .
 \]
\item The functor $\chi^c$
given by 
\[
\sigma\mapsto \bar C\otimes \Delta^{j_{\sigma}}_\ast\, .
\]
\item The tree functor $\tau$ as defined above.
\end{enumerate}
There are evident natural transformations  
\[
\mathfrak u\: \chi^c \to \chi  \qquad \text{ and }  \qquad  \mathfrak v\: \tau \to \chi\, .
\]

\begin{lem} \label{lem:conditions} The follow properties hold:
\begin{itemize}
\item The natural transformation $\mathfrak u$ is a cofibrant 
approximation to $\chi$.
\item The functor $\tau$ is positively acyclic. 
\item Both $\chi^c$ and $\tau$ are globally
bounded below by 0. 
\item The natural transformation $\mathfrak v\: \tau \to \chi$ 
induces an isomorphism on homology in degree 0
and is trivial in positive degrees.
\end{itemize}
\end{lem}

\begin{proof} The only non-trivial thing to check is that $\chi^c$ is cofibrant.
This amounts to the statement that for $j \ge 0$ the inclusion
\[
\bar C\otimes \partial \Delta^j_\ast \to \bar C\otimes \Delta^j_\ast
\]  
is a cofibration of $\Ch$, where $\partial \Delta^j_\ast$ is the simplicial chain complex
of $\partial \Delta^j$. 
But this is clear.
 \end{proof}
 
 \begin{defn} \label{defn:axiom-hypercurrent} A {\it pre-hypercurrent map} is  a natural transformation  
\[
\mathfrak j\: \chi^c \to \tau
\]
such that the induced homomorphisms in zero dimensional homology
\[
({\mathfrak v} {\mathfrak j})_\ast, \mathfrak u_\ast \: H_0\chi^c \to H_0\chi
\]
coincide.
\end{defn}

\begin{prop} \label{prop:pre-hypercurrent} A pre-hypercurrent map 
$\mathfrak j$ exists and is unique up to contractible choice.
\end{prop}

\begin{proof} Since $\mathfrak v_\ast$ is an isomorphism in degree 0, the homological condition
can be rephrased as the following equation of maps $H_\ast(\chi^c) \to H_\ast(\tau)$:
\[
{\mathfrak j}_\ast = \mathfrak v^{-1}_\ast \circ \mathfrak u_\ast\, .
\]
In particular, the value of ${\mathfrak j}_\ast$ is fixed.
The result then follows immediately  by application of
Proposition \ref{prop:lift} and Corollary \ref{cor:lift}, where the assumptions
are verified using Lemma \ref{lem:conditions}.
\end{proof}

\subsection{The hypercurrent map}
We construct the hypercurrent map as a colimit.  
Let $S_{\Sigma}$ be the full singular simplex category of $\Sigma$ 
(see Ex.~\ref{ex:simplex-cat}). 
By a subdivision argument, the inclusion of posets
\[
I_{\Sigma}\subset S_{\Sigma}
\]
induces a weak equivalence on realizations. In particular, the corresponding 
map of (singular) chain complexes
\[
I_\ast (\Sigma) \to S_\ast (\Sigma)
\]
is a chain homotopy equivalence.

The following is a special case
of a well-known result.

\begin{lem} \label{lem:constant} Let $c\: I_{\Sigma} \to \Ch$ be a constant functor with value $A_\ast$.
Then 
\[
\hocolim_{I_{\Sigma}} c \,\, \simeq\,\,  I_\ast(\Sigma) \otimes A_\ast\, ,
\]
\end{lem}

\begin{proof} The displayed homotopy colimit is just
\[
\colim_{\sigma\in I_{\Sigma}} (\Delta^{j_{\sigma}}_\ast \otimes A_\ast) \,\, \simeq\,\,  
(\colim_{\sigma\in I_{\Sigma}} \Delta^{j_{\sigma}}_\ast) \otimes A _\ast 
\]
where $\sigma$ ranges through elements of $I$. But 
\[
\colim_{\sigma\in I} \Delta^{j_{\sigma}}_\ast
\]
is just the singular chain complex of the nerve of
$I_{\Sigma}$.
\end{proof}

\begin{thm} \label{thm:exists-unique} A hypercurrent map 
\[
\cal J\: I_\ast(\Sigma) \otimes \bar C_\ast \to \bar C_\ast
\]
exists and is unique up to contractible choice.
\end{thm}

\begin{proof} Apply the colimit operation to the composed natural transformations
\[
\chi^c @> \mathfrak j >> \tau @> v >> \chi
\]
(where the pre-hypercurrent $\mathfrak j$, which is unique up to contractible choice,
is obtained from Proposition
\ref{prop:pre-hypercurrent}).
This gives a morphism of chain complexes
\[
\colim_{I_{\Sigma}} \chi^c @>>> \colim_{I_{\Sigma}} \chi\, .
\]
By Lemma \ref{lem:constant} the source of this last map is identified with 
$I_\ast (\Sigma) \otimes \bar C_\ast$, whereas the target is identified with 
$\bar C_\ast$.
\end{proof}

\subsection{Variants} We describe variants of Theorem 
\ref{thm:exists-unique} 
which are useful for the proof of the Quantization Theorem.

\begin{defn} \label{defn:posets} Let $\cal P$ be a partially ordered set. A subset $R\subset \cal P$ is {\it closed}
if for all $x\in R$ and $y\in \cal P$, then $y < x$ implies $y \in T$.
\end{defn}
Clearly, a closed subset of $\cal P$ is again a poset. 

In what follows, we set
$\cal P := I_\Sigma$, and let 
$
R \subset I_{\Sigma}
$ 
be any closed subset.
Clearly, the model category machinery 
 applies equally to the functor category $\Ch^R$.

Applying Definition \ref{defn:axiom-hypercurrent}
 to the restrictions to $T$ of the functors
$\chi$ and $\tau$, 
we obtain a restricted pre-hypercurrent map
\[
\frak j_{|R}\: \chi^c \to \tau
\]
which is unique up to contractible choice. In particular,  
we obtain the following variant of Theorem \ref{thm:exists-unique}:

\begin{cor} \label{cor:exist-unique-variant} Restriction to a closed subset
$R \subset I_\Sigma$ yields a restricted hypercurrent chain map
\[
\cal J_{|R}\: R_\ast\otimes \bar C \to \bar C
\]
which is unique up to contractible choice, where $R_\ast$ is
 the chain complex over $\Bbb F$
is freely generated in degree $j$ by the objects 
$\sigma\: \Delta^j \to \Sigma$ of $R$, with
boundary operator defined by taking the alternating sum of the
 codimension one faces of  $\Delta^j$.
\end{cor}

\subsection{A derived category approach}
An inspection of the above recipe
yields a direct construction a pre-hypercurrent map 
in the homotopy category of functors $I \to \Ch$. The idea is that
the canonical chain of natural transformations
\[
 \chi @>>> H_\ast(\tau) @<{}_\sim << \tau @>\frak v >> \chi
 \]
defines an endomorphism $[\cal J]\: \chi\to \chi$ in the homotopy category. It is trivial to check that
 the pre-hypercurrent map $\cal J$
descends to this endomorphism.

\subsection{The small cell approach} 
We briefly describe an alternative construction of a hypercurrent map when $\Sigma$ has the structure of finite CW complex. 
We say that a good protocol $\Sigma \to  \breve{\cal M}_{p,q}$ is {\it cellularly small}
if each characteristic map
$\chi\: D^j \to  \Sigma$ has image in some $\Sigma(k)$. 
This condition can usually be arranged: if $\Sigma$ is a  regular CW complex equipped with a good protocol, then there is always a subdivision of $\Sigma$ such that the protocol becomes cellularly small.  To avoid technicalities,
we will restrict ourselves to the case
when $\Sigma$ is regular, which for us is the case of interest.

Recall that a regular CW complex $Y$ is a CW complex if every characteristic map
$D^j \to Y$ is an embedding. Let $\cal P_Y$ be the poset given by the set of cells of $Y$
partially
ordered as follows: $e\le e'$ iff and only if $e$ is contained in the closure of $e'$. The realization $|\cal P_Y|$ is the barycentric subdivision of $Y$ and there is a cellular map
 $Y  @>>>	 |\cal P_Y|$ which is also a homeomorphism \cite[thms.~1.7, 2.1]{Lundell-Weingram}

Assume that $\Sigma$ is a  
regular CW complex with respect to a cellularly small protocol
$\Sigma \to  \breve{\cal M}_{p,q}$. By replacing $I_{\Sigma}$ with
$\cal P_{\Sigma}$ in the construction of the 
pre-hypercurrent map, we obtain a natural transformation
\[
\mathfrak j^{\text{cw}} \: \chi^c \to \chi
\]
just as above, but where now the source category for $\chi$ and
$\chi^c$ is the poset $\cal P_{\Sigma}$.
Taking colimits we obtain the cellular hypercurrent map
\begin{equation} \label{eqn:cellular-hypercurrent}
{\cal J}^{\text{cw}}\: C_\ast(\Sigma) \otimes \bar C_\ast \to \bar C_\ast\, ,
\end{equation}
in which $C_\ast(\Sigma)$ is the  cellular chain complex of $\Sigma$. 
When considered  in the derived category, the maps ${\cal J}^{\text{cw}}$ and $\cal J$ coincide. In particular they agree on the level of homology.

\section{Homotopy type of the space of good weights}\label{sec:good-weights}

For spaces $A$ and $B$,
 recall that the topological join $A\ast B$ is defined to be the quotient space
 of $A\times B \times [0,1]$ in which for all points
 $a,a' \in A$ and $(b,b') \in B$, the point $(a,b,0)$ is identified with $(a,b',0)$ and
 the point $(a,b,1)$ is identified with $(a',b,1)$. It is often convenient to identify
 the equivalence class of $(a,b,t)$ with the affine combination $ta + (1-t)b$.
 Another description of $A\ast B$  is 
 \[
 (CA \times B) \cup (A \times CB)
 \]
 where $CA$ is the cone on $A$ and the union is amalgamated over $CA \times CB$.

The  homotopy type of $\breve {\cal M}_{p,q}$ 
depends only on the number of cells in each dimension $\cal J$ where $j \in [p,q]$. In particular, it is independent of the choices of the attaching maps of the cells of $X$.
Recall that $X_j$ denotes the number of $j$-cells of $X$. Let
\[
E(X_j,\Bbb R)
\]
be the space of one-to-one functions $X_j \to \Bbb R$.

\begin{prop}\label{prop:good-type} 
There a preferred homotopy equivalence
\[
\breve {\cal M}_{p,q} \,\, \simeq\,\,   
E(X_p,\Bbb R) \ast E(X_{p+1},\Bbb R) \ast \cdots \ast E(X_q,\Bbb R) \, .
\]
\end{prop}

\begin{rem}\label{rem:good-type} In particular, if some $X_j$ has cardinality $\le 1$, then 
$\breve {\cal M}_{p,q}$ is contractible. 
\end{rem}

\begin{proof}[Proof of Proposition \ref{prop:good-type}]  If $q-p = 0$, 
then $\breve {\cal M}_{p,q}$ is just
$E(X_p,\Bbb R)$. 

Assume next that $q-p = 1$. Then $\breve {\cal M}_{p,q}$ decomposes as 
\begin{equation} \label{eqn:join} 
E(X_p,\Bbb R) \times F(X_q,\Bbb R) \,\, \cup\,\,  F(X_p,\Bbb R) \times E(X_q,\Bbb R)\, ,
\end{equation}
where the union is amalgamated over $E(X_p,\Bbb R) \times E(X_q,\Bbb R)$ and
$F(X,\Bbb R)$ denotes the function space of maps $X\to \Bbb R$.  As the latter is
a convex space, there is a preferred homotopy equivalence of pairs
\[
(F(X,\Bbb R),E(X,\Bbb R)) \simeq (C(E(X,\Bbb R)),E(X,\Bbb R))\, .
\]
Hence, the union \eqref{eqn:join} is identified with the join
$E(X_p,\Bbb R) \ast E(X_q,\Bbb R)$. This completes the case $q-p = 1$.

The general case  now follows from a straightforward induction on $q-p$ which we leave to the reader.
\end{proof}

\begin{cor} If $X$ has at most one $j$-cell for some
$j \in [p,q]$, then the topological hypercurrent map
$\cal J\: H_{p-q}(\Sigma) \otimes H_p(X) \to H_q(X)$ is trivial
for any  good protocol $\Sigma$.
\end{cor}

\begin{proof} If we take
$\Sigma = \breve {\cal M}_{p,q}$, then $\Sigma$ is contractible by Proposition \ref{prop:good-type} (since at least
one of the factors $E(X_j,\Bbb R)$ is contractible for some $j$).   This shows gives the result in the universal case. The general case follows by functoriality.
\end{proof}

Let $n_j$ be the cardinality $X_j$  and set
\[
c_{p,q} := \prod_{j = p}^{q} (n_j!-1)
\]

\begin{cor} Assume $n_j \ge 2$ for $j \in [p,q]$. then
there is a preferred homotopy equivalence  
\begin{equation} \label{eqn:decomposition}
\breve {\cal M}_{p,q} \,\, \simeq \,\, \bigvee_{c_{p,q}} S^{q-p} \, ,
\end{equation}
i.e., $\breve {\cal M}_{p,q}$ has the homotopy type of a $c_{p,q}$-fold wedge of $(q-p)$-spheres.
\end{cor}

\begin{proof} The proof uses the following elementary fact: let
$A$ and $B$ be well-pointed based spaces. Then the join $A\ast B$  has the homotopy type
of $\Sigma A \smsh B$, that is, the suspension of the smash product.

Let $L(X_j)$ denote the set of linear orderings on the set $X_j$.
Then the obvious map $E(X_j,\Bbb R) \to L(X_j)$ is a deformation retraction.
In particular, $E(X_j,\Bbb R)$ has the homotopy type of a finite set of cardinality
 $n_j!$. Assuming that $n_j \ge 2$, a choice of basepoint in $E(X_j,\Bbb R)$ 
gives an identification up to homotopy
\[
E(X_j,\Bbb R) \simeq \bigvee_{n_j!-1} S^0\, ,
\]
since the right side has cardinality $n_j!$.
The result then follows from Proposition \ref{prop:good-type}, using the fact
that smash products distribute over wedges.
\end{proof}

\subsection{Explicit description of the summands}
We describe  maps 
\[
S^{q-p} \to \breve {\cal M}_{p,q}
\] 
that correspond to
 the summands in the decomposition \eqref{eqn:decomposition}.
For this we fix a basepoint $\ast_j \in E(X_j,\Bbb R)$ for each $j\in [p,q]$. 
Note that the automorphism (permutation)
group of $X_j$ acts freely on on $E(X_j,\Bbb R)$.
For each non-trivial automorphism $\sigma_j \: X_j \to X_j$, one 
has a map 
\[
\alpha_{\sigma_j} \: S^0 \to E(X_j,\Bbb R)
\] 
in which $\alpha_{\sigma_j}(-1)=\sigma\cdot \ast_j$ and $\alpha_{\sigma_j}(+1) = \ast_j$. 
For each $\cal J$ we choose such a $\sigma_j$ and join the maps $\alpha_\sigma$ together. 
This produces a map
\[
S^{q-p} \cong S^0 \ast  \cdots \ast S^0 @> \alpha_{\sigma_p} \ast \cdots \ast
\alpha_{\sigma_q}>> E(X_p,\Bbb R) \ast \cdots \ast E(X_q,\Bbb R) \simeq 
\breve {\cal M}_{p,q}\, 
\]
yielding a summand in the decomposition. By varying the choices of the automorphisms, every summand in the decomposition arises
in this way.

\section{Proof of Theorems \ref{bigthm:properties} and  \ref{bigthm:examples}} \label{sec:proof-props-ex}

Let $X = S^q$ with the CW structure having two $j$-cells $e^j_\pm$ for $0 \le j \le q$.
The attaching maps $\partial e^j_\pm \to S^{j-1} = X^{(j-1)}$ are given by the identity map, where $0 \le j \le q$. Similarly, $Y = D^q \vee S^q$, where the
CW structure is such that $Y^{(q-1)} = X^{(q-1)} = S^{q-1}$
and the two $q$-cells $e^q_\pm$ 
are attached respectively by the identity and the constant map. Note that
$X$ and $Y$ have the same number of cells in each degree and have the same homotopy type.

\begin{proof}[Proof of Theorem \ref{bigthm:examples}]
Consider first the case $q=1$. For convenience, we change notation
and denote the weights
by $E = (E_-,E_+)$ and $W = (W_-,W_+)$, where 
$E_\pm$ is a real number associated with $e^0_\pm$ and 
$W_\pm$ is a real number associated with $e^1_\pm$.
A spanning $0$-tree is then a $0$-cell and a spanning 1-tree is 
a $1$-cell. The greedy algorithm says in this case that 
 we select the $0$-tree $e^0_+$ if $E_+ < E_-$ and $e^0_-$ if $E_+ > E_-$.
 Similarly we select the $1$-tree $e^1_+$ if $W_+ < W_-$ and we select 
 the $1$-tree $e^1_-$ 
 if $W_- < W_+$.  Suppose we can find a map 
 $f\:S^1\to \breve {\cal M}_{0,1}$, with $f(e^{i\theta}) = (E_-(\theta),E_+(\theta),W_-(\theta),W_+(\theta))$, having the properties indicated by 
 Fig.~\ref{fig:square}, where we are considering $S^1$ as the boundary of the square using the homeomorphism
 $(x,y) \mapsto \tfrac{(x,y)}{\|(x,y)\|}$ for $(x,y)$ lying
 in the boundary of the square $[-1,1]^{\times 2}$.
\begin{figure}
\centering
 \begin{tikzpicture}[thick, scale=.75]
 \draw[red, line width=2] (0,0) -- (4,0);
  \draw[green!40!black,line width=2] (4,0) --(4,4);
  \draw[red, line width=2] (4,4) -- (0,4);
  \draw[green!40!black,line width=2] (0,4) -- (0,0);
 \node at (2.3,4.3) {\tiny $W_+ < W_-$};
 \node at (2.3,-0.3) {\tiny $ W_+ > W_- $};
 \node at (4.9,2) {\tiny $ E_+ > E_-$};
 \node at (-0.9,2) {\tiny $ E_+ < E_-$};
 \end{tikzpicture}
 \caption{}\label{fig:square}
 \end{figure}
 The contribution to the current on the green edges of the square is trivial, since
 along each such edge there is a preferred $0$-tree.
 The top edge of the square gives a contribution of $e^1_+$ to the current and the bottom
 edge contributes $e^1_-$. It follows that the current generated by
 the square is the 1-cycle $e^1_+ + e^1_-$.
Hence $\cal J^{\text{cw}} \: H_1(S^1) \otimes
H_0(S^1) \to H_1(S^1)$ is non-trivial if we can find such a map $f$.
A concrete example is given by $f(e^{i\theta}) = ((0,\cos \theta),(0,\sin\theta))$.
This takes care of the case $q=1$.

Next consider the case $q=2$. We identify $S^2$ with the boundary of 
the cube $[-1,1]^{\times 3}$
and suppose we can find a map $f\: S^2 \to \breve{\cal M}_{0,2}$ with the
 properties appearing in Fig.~\ref{fig:cube}.
\begin{figure}
\centering
\begin{tikzpicture}[thick,scale=2.8] 
    \coordinate (A1) at (0, 0);
    \coordinate (A2) at (0, 1);
    \coordinate (A3) at (1, 1);
    \coordinate (A4) at (1, 0);
    \coordinate (B1) at (0.3, 0.3);
    \coordinate (B2) at (0.3, 1.3);
    \coordinate (B3) at (1.3, 1.3);
    \coordinate (B4) at (1.3, 0.3);

    \draw[very thick] (A1) -- (A2);
    \draw[very thick] (A2) -- (A3);
    \draw[very thick] (A3) -- (A4);
    \draw[very thick] (A4) -- (A1);

    \draw[dashed] (A1) -- (B1);
    \draw[dashed] (B1) -- (B2);
    \draw[very thick] (A2) -- (B2);
    \draw[very thick] (B2) -- (B3);
    \draw[very thick] (A3) -- (B3);
    \draw[very thick] (A4) -- (B4);
    \draw[very thick] (B4) -- (B3);
    \draw[dashed] (B1) -- (B4);

    \draw[fill=green,opacity=0.5] (A1) -- (B1) -- (B4) -- (A4);
    \draw[fill=black!20,opacity=0.8] (A1) -- (A2) -- (A3) -- (A4);
    \draw[fill=red,opacity=0.5] (A1) -- (A2) -- (B2) -- (B1);
    \draw[fill=black,opacity=0.5, ] (B1) -- (B2) -- (B3) -- (B4);
    \draw[fill=red,opacity=0.5] (A3) -- (B3) -- (B4) -- (A4);
    \draw[fill=green,opacity=0.5] (A2) -- (B2) -- (B3) -- (A3);
\node [rotate=90,scale = .65]  at (1.2,.7) {$W_{1+} > W_{1-}$ };
\node [rotate=90,scale = .65] at (0.15,.65) {$W_{1-} > W_{1+}$ };
\node[scale = .65] at (.75,1.15) {$W_{2+} > W_{2-}$ };
\node[scale = .65] at (.65,.15) {$W_{2-} > W_{2+}$ };
\node[scale=.65, opacity=.4] at (.73,.75) {$W_{0-} > W_{0+}$ };
\node[scale=.65] at (.58,.55) {$W_{0+} > W_{0-}$ };
\end{tikzpicture}
\caption{} \label{fig:cube}
\end{figure}
Then an unravelling of the  definition of $\cal J^{\text{cw}}$ shows
 the contribution to the hypercurrent is trivial
along all faces of the cube except the top face, which contributes $e^2_-$ and the bottom face, which contributes $-e^2_+$. Hence the value of the hypercurrent on the 
entire boundary of the cube is given by the cycle $e^2_- - e^2_+$, which generates 
homology of $X = S^2$ in degree 2.  An example of $f\: S^2 \to 
\breve{\cal M}_{0,2}$ with the above properties is
\[
f(x_0,x_1,x_2) = (0,x_0,0,x_1,0,x_2)\, ,
\]
where $\sum x_i^2 = 1$, and $W_{i+} = x_i,W_{i-} = 0$. This shows that the hypercurrent map
$\cal J^{\text{cw}}\: H_2(S^2) \otimes H_0(S^2) \to H_2(S^2)$ is non-trivial in this case.

The case $q > 2$ is completely analogous: we take the boundary of 
a $q$-dimensional cube and label its $q+1$ pairs of opposing faces by
$W_{j-} < W_{j+}$ and $W_{j+} < W_{j-}$ for $0\le j \le q$. The map
 $f\: S^q \to \breve {\cal M}_{0,q}$ is given by 
\begin{equation} \label{eqn:ident-weights-sphere}
f(x_0,\dots,x_q) = (0,x_0,0,x_1,\dots,0,x_q)\, .
\end{equation}
Again, the hypercurrent map will be non-trivial. 

We remark that
$f$ in the above example is a homotopy equivalence.

Consider now the case of $Y = D^q \vee S^q$ with the CW structure having
the same $(q-1)$-skeleton as $X$, i.e., $S^{q-1}$. There are two 
two $q$-cells $e^q_+$ and $e^q_-$, where
$e^q_-$ is attached via the identity map and $e^q_+$ is attached using the constant
map whose value is a 0-cell. Then there is exactly one $q$-tree $T$ given by
the summand $D^q  \subset Y$. It follows that the hypercurrent is necessarily trivial since it factors through the homology of $D^q$.
\end{proof}

\begin{proof}[Proof of Theorem \ref{bigthm:properties}]
Statements (1)-(3) of Theorem \ref{bigthm:properties} follow directly from the construction of the pre-hypercurrent map.  It will therefore suffice to extablish statement (4).
When $p = 0$ an example satisfying (4) is given by Theorem \ref{bigthm:examples}. Therefore,
it will suffice to consider the case when $p > 0$.

Let $Z = X/X^{(p-1)}$, where $X = S^q$ is provided with the CW structure of Theorem \ref{bigthm:examples}. In addition to a unique 0-cell,
 $Z$ has two cells in each dimension $j \in [p,q]$.
Furthermore, $Z$ has the homotopy type of $S^{p} \vee S^q$. Hence, $[p,q]$ is a gap
(for both $X$ and $Z$).
There is a homeomorphism
\[
\breve {\cal M}_{0,q}(S^{q-p}) \to \breve {\cal M}_{p,q}(Z)\, ,
\]
where $S^{q-p}$ is regarded as the $(q-p)$-skeleton of $X$, which 
is given by the evident bijections on the set of cells.
The homotopy equivalence $S^{(q-p)} @> f > {}^\sim >\breve {\cal M}_{0,q}(S^{q-p})$
defined as in \eqref{eqn:ident-weights-sphere}, followed by this homeomorphism gives
a good protocol
\[
\Sigma := S^{q-p} @> {}_{\sim} >>\breve {\cal M}_{p,q}(Z)
\]
The rest of the argument follows by the argument of
the proof of Theorem \ref{bigthm:examples}(1).
\end{proof}

\section {Robust weights}\label{sec:robust}

\begin{defn} The {\it discriminant} $\cal D_{p,q}$ is the complement
of the space good parameters from $\cal M_{p,q}$, i.e.,
\[
\cal D_{p,q} := \cal M_{p,q} \setminus \breve {\cal M}_{p,q}\, .
\]
\end{defn}

Let $\cal M_{p,q}^+$ be the one-point compactification of 
$\cal M_{p,q}$ and similarly, let $\cal D^+_{p,q}$ be the one-point compactification of 
$\cal D_{p,q}$. Then $\cal M_{p,q}^+$ is a sphere of dimension $\sum_j n_j$.

\begin{thm} \label{thm:cw-structure} $\cal M_{p,q}^+$ has the structure
of a regular CW complex in which the subspace
 $\cal D_{p,q}^+$ is a subcomplex of codimension $q-p+1$.
 \end{thm}
 
The proof of Theorem \ref{thm:cw-structure} 
is an adaptation of the argument of \cite[prop.~7.1]{CKS13}.

By {\it height data} $h_\bullet$ we mean a partial ordering
$<_i$ on  $X_i$ for $i = p,\dots,q$. The partial ordering defines
an equivalence relation $\sim_i$ defined by $x\sim_i y$ if and only
if the length of a maximal chain terminating in $x$ is the same 
as the length of a maximal chain terminating in $y$.
Note that the set $X_i/\!\!\sim_i$ comes equipped with a 
canonical linear order defined by $<_i$.

Here is the main example:

\begin{ex} \label{ex:height} For $W_\bullet := (W_p,\dots,W_q) \in \cal M_{p,q}$ we define
$x <_i y$ if and only if $W_i(x) < W_i(y)$ for $x,y \in X_i$.
\end{ex}

\begin{defn} \label{defn:cell} Given height data $h_\bullet$, let
\[
C(h_\bullet) \subset \cal D_{p,q}
\]
be the subspace consisting of the set of weights $W_\bullet$ 
whose associated height data is $h_\bullet$ in the sense
of Example \ref{ex:height}.

Let $D(h_\bullet)$ denote the closure of $C(h_\bullet)$ in $\cal M_{p,q}$.
\end{defn}

\begin{prop}\label{prop:cell} The one-point compactification
$D(h_\bullet)^+$ of $D(h_\bullet)$ is homeomorphic to a disk of dimension $\sum_i d_i$,
where $d_i$ is the cardinality of $X_i/\!\!\sim_i$.
\end{prop}

\begin{proof} There is an evident homeomorphism
\[
D(h_\bullet) \cong D(h_p) \times \cdots \times D(h_q)\, ,
\]
where $D(h_i)$ is the closure of the set of those $W_i\:X_i \to \Bbb R$ in the function
space $F(X_i,\Bbb R)$
that induce the equivalence relation $\sim_i$. 
Hence, it will be enough to show that
$D(h_i)^+$ is homeomorphic to a disk of dimension $d :=d_i$.

We first show that $D(h_i)$ is homeomorphic to the subspace space $H\subset \Bbb R^d$
consisting of $d$-tuples of real
numbers
\[
(x_1,\dots,x_{d})
\]
such that $x_1 \le x_2 \le \cdots \le x_d$. 
The homeomorphism is defined as follows: a function $W_i \to X_i \to \Bbb R$
in the interior $C(h_i)$ of $D(h_i)$
defines a $d$-tuple of increasing real numbers $x_1 < \dots < x_d$ by taking the
image of $W_i$. This assignment defines an embedding $C(h_i) \to \Bbb R^d$ which extends
to a homeomorphism $D(h_i) @> \cong >> H$. On the other hand, it is easy to show that
$H^+$ is homeomorphic to the disk $D^d$.
 \end{proof}
 
 \begin{proof}[Proof of Theorem \ref{thm:cw-structure}]
 The inclusions $D(h_\bullet)^+ \subset \cal M^+_{p,q}$ define characteristic maps
 for a cell structure on $\cal M_{p,q}$. It is straightforward to check that cells
 of a given dimension are attached on top of cells of lower dimension,
 so we obtain regular CW structure on
 $\cal M^+_{p,q}$. The latter is a sphere of dimension $\sum_i n_i$, where
 $n_i = |X_i|$ for $p \le i\le q$.
 
 Clearly, $\cal D^+_{p,q} \subset \cal M^+_{p,q}$ is a subcomplex, since
a weight system $W_\bullet$ lies in the discriminant 
 precisely when no $W_i\: X_i \to \Bbb R$
 is one-to-one (this means the associated partial ordering is not linear).
 A straightforward dimension count then shows that the top dimensional cells of $\cal D^+_{p,q}$ have dimension
 \[
 \sum_{i=p}^q (n_i-1)\, .
  \] 
It follows that the codimension  of $\cal D^+_{p,q}$ is $q-p+1$.
 \end{proof}
 
 \subsection{The space of robust weights}
Consider a top-dimensional cell $D\subset \cal D_{p,q}^+$.  
Choose a small  closed disk $E^{q-p+1} \subset \cal M_{p,q}$
which is transversal to $D$ such that $D \cap E$ is a single point.
Then we obtain an inclusion 
\[
S^{q-p} := \partial E \to \breve{\cal M}_{p,q}\, .
\]
Consequently, the hypercurrent 
\[
{\cal J}^\ast_{p,q} \in H^{q-p}(\partial E; H^p(X;H_q(X))) \cong 
\hom(H_p(X), H_q(X))
\]
is defined, and it doesn't depend on the choice of
$E$.
To emphasize its dependence on 
the cell $D$, we denote the above hypercurrent by $J_D$.

\begin{defn} The cell $D$ is {\it inessential} if $J_D$ is trivial.
It is {\it essential} if $J_D$ is non-trivial.
 \end{defn}
 
\begin{defn} Let $\check {\cal D}_{p,q} \subset \cal D_{p,q}$ be
the subspace obtained by removing the interiors of every inessential cell.
The space of {\it robust weights} $\check {\cal M}_{p,q}$ is defined as the complement of 
$\check {\cal D}_{p,q}$ in $\cal M_{p,q}$, i.e.,
\[
\check {\cal M}_{p,q} \,\, := \,\, \cal M_{p,q} \setminus \check {\cal D}_{p,q}\, .
\]
A {\it robust protocol} is a map of spaces $\Sigma \to  \check {\cal M}_{p,q}$. 
\end{defn}
 
 \begin{prop}  There is a preferred extension
 \[
 \check {\cal J}_{p,q}\: H_{q-p}(\check{\cal M}_{p,q})\otimes H_p(X)\to H_q(X)
 \]
of the universal topological hypercurrent homomorphism $\cal J$.
  \end{prop}
  
 \begin{proof} This is a direct consequence
 of obstruction theory, since the relative homology groups
 $H_\ast(\check{\cal M}_{p,q}, \breve{\cal M}_{p,q})$ 
 are non-trivial in the single degree $q-p$ and are
 freely generated by the homology class of the maps
 $(E,\partial E) \to (\check{\cal M}_{p,q},\breve{\cal M}_{p,q})$, 
 with one such map for each inessential cell $D$. Since each $J_D$ is trivial
 it follows that the extension exists. The extension is preferred because 
  the relative homology group vanishes in degree $q-p+1$.
 \end{proof}
 
 \begin{cor} A robust protocol $\Sigma\to \check{\cal M}_{p,q}$ induces
  a hypercurrent homomorphism
 \[
 \check {\cal J}_{p,q}\: H_{q-p}(\Sigma) \otimes H_p(X) \to H_q(X)\, .
 \]
 \end{cor} 
 
 \subsection{Homotopy type of $\check{\cal M}_{p,q}$}
An explicit closed disk $E \subset \breve{\cal M}_{p,q}$ is obtained as follows:
let $W_\bullet \in \cal D^+_{p,q}$ be a point in the interior of a top dimensional
closed cell. Then for each $i$ with $p \le i \le q$, there are a unique pair of elements
 $x_i,y_i \in W_i$
such that $W_i(x_i) = W_i(y_i)$. 
Then construct a transversal cube of dimension $q-p+1$ in $\check{\cal M}_{p,q}$ with center $W_\bullet$, we resolve the $W_i$ as follows: consider the set 
 of weights $V_\bullet = (V_p,\dots, V_q)$
such that for a sufficiently small choice of $\epsilon > 0$,
\begin{itemize}
\item $V_i(u) = W_i(u)$ for $u \ne y_i$,
\item $|V_i(y_i) - W_i(y_i)| \le \epsilon$. 
\end{itemize}
(compare Fig.~\ref{fig:cube}). 
Then the set of such $V_\bullet$ forms a  transversal $(q-p+1)$-cube.
Let $\tau\: S^{q-p} \to \breve{\cal M}_{p,q}$ denote the restriction to the boundary.

Applying the above procedure to each inessential cell, we obtain
a map
\[
\beta\: S^{q-p} \times T \to \breve {\cal M}_{p,q}\, ,
\]
where $T$ denotes the set of inessential cells of ${\cal D}_{p,q}$.
By Alexander duality, we infer

\begin{lem} There is a homotopy equivalence
\[
\check {\cal M}_{p,q} \,\, \simeq \,\, \breve {\cal M}_{p,q} \cup_{\beta} (D^{q-p+1} \times T) \, .
\] 
\end{lem}

We now identify the map $\tau$ in terms of the decomposition of $\breve {\cal M}_{p,q}$
into a wedge of spheres.
Observe that $\tau$ is just the iterated join of maps
\[
S^0 \ast \cdots S^0 @>\tau_p \ast \cdots \ast \tau_q >> E(X_p,\Bbb R) \ast \cdots \ast E(X_q,\Bbb R)\, ,
\]
where $\tau_j\: S^0 \to E(X_j,\Bbb R)$ is defined by
$\tau_j(\pm 1) = W^{\pm}_j$, with $W^{\pm}_j(x) = W_j(x)$ for $x \ne y_j$
and $W^{\pm}_j(y_j) = W_j(y_j) \pm \epsilon$. 

The map $\tau_j$ is generally not pointed: it is determined up to contractible choice
by two nontrivial permutations, given by the changes in the ordering from the basepoint component to the components of $\tau_i(\pm 1)$. Hence, in terms of the identification
$\breve{\cal M}_{p,q}$ with a wedge of $(q-p)$-spheres \eqref{eqn:decomposition},
$\tau$ is identified with the following composite 
\begin{equation} \label{eqn:attach}
S^{q-p} @> \text{pinch} >> \bigvee_{2^{q-p}} S^{q-p} \subset \bigvee_{c_{p,q}} S^{q-p} \, ,
\end{equation}
where the first of these is the pinch map and the second is an inclusion of summands
obtained from the choice of two non-trivial permutations of $X_j$ for each $j$, with $j \in [p,q]$.  

\begin{cor} There is a homotopy equivalence
\[
\check {\cal M}_{p,q} \,\, \simeq \,\, \bigvee_{d_{p,q}} S^{q-p}\, ,
\]
with $d_{p,q} = c_{p,q} - u$, where $u$ is the number of inessential cells.
\end{cor}

\begin{proof}  $\check {\cal M}_{p,q}$ is obtained by attaching $(q-p+1)$-cells
a wedge of $c_{p,q}$ copies of the sphere $S^{q-p}$, where each cell is attached
by a map of the form \eqref{eqn:attach}. 

The mapping cone of \eqref{eqn:attach}
 is the wedge of  the mapping cone of 
 the pinch map with a wedge of spheres $S^{q-p}$ indexed
 over $c_{p,q} - 2^{q-p}$. Hence, by induction it suffices to show that the mapping cone
 of the above pinch map is a wedge of $2^{q-p}-1$ copies of $S^{q-p}$. But this follows directly from Lemma \ref{lem:pinch} below.
  \end{proof}

The above proof used the following assertion about the pinch map. 
 
 \begin{lem} \label{lem:pinch} The mapping cone of the pinch map
 \[
 S^k \to \bigvee_m S^k \, .
 \]
 is a wedge of $(m-1)$-copies of $S^k$. 
 \end{lem}
 
\begin{proof}  For the sake of completeness we first define the pinch map. Let $T\subset S^1$ be the set of $m$-th roots of unity.
Suspend the quotient map $S^1 \to S^1/T$
$(k-1)$-times to obtain a map $S^k\to \Sigma^{k-1}(S^1/T)$. The latter is 
the pinch map once we identify $\Sigma^{k-1}(S^1/T)$ with an $m$-fold wedge of 
$k$-spheres (this is clear, since $S^1/T$ is a bouquet of $m$ circles).

We now proceed with the proof. Consider the commutative diagram
 \[
 \xymatrix{
 \ast \ar[d] \ar[r] & \bigvee_{m-1} S^k \ar@{=}[r] \ar[d]^{i_1} & \bigvee_{m-1} S^k \ar[d] \\
 S^k \ar[r]^{\text{pinch}} \ar@{=}[d] & \bigvee_m S^k \ar[r] \ar[d]^{p_1} & C \ar[d] \\
 S^k \ar@{=}[r] & S^k \ar[r] & \ast
 }
 \]
 whose horizontal and vertical rows are homotopy cofiber sequences, where $p_1$ is projection onto the first summand and $i_1$ is the inclusion away from the first summand. The upper left square is homotopy cocartesian, so the map $\vee_{m-1} S^k \to C$ is a homotopy equivalence.
\end{proof}

\section{Analytical hypercurrents \label{section:analytical}}

In this section we define an analytical version of the hypercurrent. 
In what
follows, $\Bbb F$ is taken to be the field of real numbers $\Bbb R$.

\subsection{Motivation} Suppose $\Sigma = \Bbb R$ 
and $X$ is a connected CW complex of dimension one. 
Let $C_\ast(X)$ be the real cellular chain complex of $X$.
Fix
a smooth protocol $\gamma\: \Bbb R \to \cal M_X$.
The {\it master equation} (or Kolmogorov equation) is a first order linear differential equation
\begin{equation}\label{eqn:master}
\dot\rho = -H\rho 
\end{equation}
where $H = H(t) = \partial \partial^\ast_{E,W}$ is the {\it biased} graph Laplacian, and $\rho(t) \in C_0(X)$ is a one-parameter family of 
zero chains, i.e., distributions.
The operator  $\partial^\ast_{E,W}\: C_0(X) \to C_1(X) $ is the (time-dependent) biased formal adjoint to the boundary operator 
$\partial\: C_1(X) \to C_0(X)$ that uses 
$\gamma(t) = (E(t),W(t))$ to modify the standard inner product structure on $C_\ast(X)$ (for details, see \cite{CKS13} and \cite{CCK15b}).
Equation \eqref{eqn:master} describes biased diffusion
on $X$.

The {\it analytical current} $\gamma$ is the quantity
\begin{equation} \label{eqn:ast-defn}
-\partial^\ast_{E,W}\rho \in C_1(X)
\end{equation}
where $\rho$ is a formal solution to \eqref{eqn:master}. 
Then  \eqref{eqn:ast-defn} coincides with
\begin{equation} \label{eqn:dagger-definition}
\partial^{\dagger}_{E,W} \dot \rho \,\, = \,\, \partial^{\dagger}_{E,W} d\rho \lrcorner \mu
\end{equation}
where $\partial^{\dagger}_{E,W}$ is the Moore-Penrose pseudoinverse
of $\partial$ (using the modified inner product structure), 
$d$ is exterior derivative, $\mu = \tfrac {d}{dx}$ is the co-volume form of $\Bbb R$ and $\lrcorner$ denotes contraction.  

Since the operation $\phi \mapsto \phi\lrcorner \mu$  equates $1$-forms with $0$-forms,  
we could have instead defined the analytical current 
as the 1-form
\begin{equation} \label{eqn:J-as-1-form}
J^{\text{an}} := \partial^{\dagger}_{E,W} d\rho \,\, \in\,\,  
\Omega^1(\Bbb R;C_1(X))
\end{equation} 
The main advantage of the reformulation \eqref{eqn:J-as-1-form} is that
it will generalize to higher dimensions.

\begin{rem}
In \cite{CKS13} we studied the case when $\gamma$ is periodic. After taking adiabatic (slow driving) limit, we showed that the formal solution $\rho$
to \eqref{eqn:master} is also periodic. In this case, we can consider 
the analytical current as a $1$-form on the circle. Furthermore, 
the ``adiabatic theorem'' implies that can 
use the Boltzmann distribution in place of 
the formal solution $\rho$ to define $J^{\text{an}}$. The Boltzmann distribution is the unique 0-form $\rho^B\: S^1 \to C_{0}(X)$  
such that  $\partial^\ast_{E,W}\rho^B = 0$.

The fundamental theorem of calculus implies that integration of this one form defines a homology class
$q\in H_1(X;\Bbb R)$ called the {\it average current}.
In what follows below,  the periodicity condition is replaced by 
the assumption the domain $\Sigma$ of $\gamma$ is a closed manifold.
\end{rem}

\subsection{Higher dimensions}
The input for the construction is a triple
\[
(X,\Sigma,\gamma)\, ,
\]
in which $X$ is a finite connected CW complex, and $\Sigma$ is a compact smooth manifold and $\gamma\: \Sigma \to  \prod_j \Bbb R^{X_j}$ is a smooth protocol
with respect to a gap $[p,q]$. 
As usual, we often write $\gamma(b) = W_\bullet(b)$.
We do not require $\gamma$ to be good.

The analytical hypercurrent takes values in homology with real coefficients. 
We will define a smooth differential form
\[
\cal J^{\text{\rm an}}  \in \Omega^\ast(\Sigma;\End(\bar C)) 
\]
of total degree zero. 

\subsection{Generalized inverses} 
We  digress to describe formulas for the pseudoinverse of a linear transformation in two different cases.

Suppose that $A\: V\to W$ is a one-to-one linear transformation of finite dimensional real vector spaces, where $W$ is equipped with an inner product.
The adjoint can be considered as the linear map $A^\ast\: W \to V^\ast$ characterized by the equation $A^\ast(w)(v) = \langle w,Av\rangle$. Then
the composition $A^\ast A$ is invertible and the pseudoinverse of $A$ in this case
can be defined as the left inverse 
\[
A^\dagger := (A^\ast A)^{-1}A^\ast\: W\to V \, .
\]
In this instance one can characterize $A^\dagger$ as follows: for a given vector $w\in W$, the vector $v := A^\dagger(w)\in V$ 
is the unique vector which minimizes
the norm
\[
|Av - w|\, .
\]
Similarly, suppose a linear transformation $D\: V \to U$ is onto, 
and assume that $V$ is equipped with an inner product. Consider
the adjoint $D^\ast\: U^\ast \to V$ that is characterized by
$\langle D^\ast(\delta),v\rangle = \delta(D(v))$. In this case $DD^\ast$ is invertible and we define the pseudoinverse as the right inverse
\[
D^\dagger := D^\ast (DD^\ast)^{-1}\: U \to V\, .
\]
Then $D^\dagger$ is characterized by the following property:
for any vector $u \in U$, the vector $v = D^\dagger(u)\in V$ is the unique vector
satisfying $Dv = u$ such that the norm
\[
|v|
\]
is a minimum.

\subsection{Some operators}
Fix a smooth protocol $\Sigma \to \cal M_{p,q}$.
For each $b\in \Sigma$ one has a smooth weight $W_k(b)\: X_k \to \Bbb R$ for $p \le k \le q$.
The weight
$W_k(b)$ defines a modified inner product on $\bar C_k$ which is defined on basis
elements by $\langle e,e'\rangle_W := W_k(b)\delta_{ee'}$ where $e,e' \in X_{k-p}$ and
$\delta_{ee'}$ is Kronecker delta. 

Using the modified inner product, we form the pseudoinverse of the surjection 
$\partial\: \bar C_k \to \bar B_{k-1}$, where $B_j$ denotes the vector space of $j$-boundaries.
If we let $b\in B$ vary, the construction just outlined gives a 0-form
\[
\partial_W^{\dagger}\: \Sigma \to  \hom(\bar B_{k-1},\bar C_k)\, .
\]
The latter  induces a homomorphism of the same name
\begin{equation} \label{eqn:partial-pseudo}
\partial_W^{\dagger} \: \Omega^\ast(\Sigma;\hom(U,\bar B_{k-1})) \to
\Omega^\ast(\Sigma;\hom(U,\bar C_k))
\end{equation}
for any finite dimensional vector space $U$. The map \eqref{eqn:partial-pseudo}
can be described pointwise at $b\in \Sigma$ by the operation which
sends $\phi_b \: \Lambda^\ast T_b\Sigma 
\to \hom(U,\bar B_{k-1})$ to the composite
\[
\Lambda^\ast T_b\Sigma 
@> \phi_b >> \hom(U,\bar B_{k-1}) @>\partial_W^{\dagger}(b)\circ{-} >>
\hom(U,\bar C_k)\, ,
\]
where the second displayed map is given by $f\mapsto \partial_W^{\dagger}(b)\circ{f}$.

Similarly, using the modified inner product, we can form the pseudoinverse of 
the inclusion $i\: \bar B_k \to \bar C_k$ to obtain a $0$-from
\begin{equation} \label{eqn:i-dagger}
i_W^\dagger\: \Sigma \to  \hom(\bar C_k,\bar B_k)\, .
\end{equation}

\subsection{Procedure} 
For each integer $\ell \ge 0$ we construct
an $\ell$-form
\[
\cal J^\text{an}_\ell \in \Omega^{\ell}(\Sigma; \hom(\bar C_0,\bar C_{\ell}))\, .
\] 
\noindent Step (1): When $\ell = 0$, define 
\[
\cal J^{\text{\rm an}}_0 \:\Sigma \to \End(\bar C_0,\bar C_0) = 
\]
to be 
$I - i_W^\dagger\: \Sigma \to \hom(\bar C_0,\bar C_0)$, where, $I$ denotes the constant map whose value is the identity.
 \medskip

\noindent Step (2): When $\ell = 1$, note that
\[
d\cal J^{\text{\rm an}}_0 = dI - d i_W^\dagger = -di_W^\dagger
\]
lies in the subspace $\Omega^1(\Sigma; \hom(\bar C_0,\bar B_0))$. Define
\[
\cal J^{\text{\rm an}}_1 := \partial_W^\dagger d \cal J^{\text{\rm an}}_0\, ,
\]
where $\partial_W^\dagger$ is as in \eqref{eqn:partial-pseudo}. Note that
$\cal J^{\text{\rm an}}_1$ lies in $\Omega^1(\Sigma; \hom(\bar C_0,\bar C_1))$.
\medskip 
 
\noindent Step (3):  Assume 
\[
\cal J_{\ell-1}^\text{an} \in \Omega^{\ell-1}(\Sigma; \hom(\bar C_0,\bar C_{\ell-1}))
\]
has been defined for some $\ell$ satisfying $2\le \ell \le q-p$.
We will show how to define $\cal J_{\ell}^\text{an}$.

First apply exterior derivative to $\cal J_{\ell-1}^\text{an}$, giving
\[
d\cal J_{\ell-1}^\text{an} \in 
\Omega^\ell(\Sigma; \hom(\bar C_0,\bar C_{\ell-1}))\, .
\] 

\begin{lem} Assume $2\le \ell \le q-p$. 
Then $d\cal J_{\ell-1}^\text{an}$
lies in the subspace
\[
\Omega^\ell(\Sigma;  \hom(\bar C_0,\bar Z_{\ell-1}))\, .
\]
\end{lem}

\begin{proof} For any chain $c \in \bar C_\ast$ we write 
$\cal J_{\ell-1}^\text{an}(c) \in 
\Omega^{\ell-1}(\Sigma; \bar C_{\ast + \ell-1})$
for the evaluation of $\cal J_{\ell-1}^\text{an}$ at $c$
arising from the linear transformation
\[
\End(\bar C)_{\ell-1} \to \bar C_{\ast + \ell-1}
\]
 given by $f\mapsto f(c)$.

It will be enough to prove that $\partial d\cal J_{\ell-1}^\text{an}(c) = 0$ for all $c$, 
where the boundary operator $\partial$ in this case
is the one for $\bar C_\ast$. The latter 
assertion follows directly from the fact that the operators $\partial$ and
$d$ commute and the fact that $\partial \partial^{\dagger}_W$ is the identity.
\end{proof}
 
Since $[p,q]$ is a gap, we have $\bar B_{\ell-1} = \bar Z_{\ell-1}$
 for $2 \le \ell \le q-p$. Hence,

 \begin{cor} If $2 \le \ell \le q-p$, then the
  form $d\cal J_{\ell-1}^\text{an}$ lies in the subspace
 \[
\Omega^\ell(\Sigma;  \hom(\bar C,\bar B)_{\ell-1})\, .
\]
\end{cor}
 
Using the corollary, we are entitled to apply 
\[
\partial_W^\dagger\: \Omega^\ell(\Sigma;  \hom(\bar C_0,\bar B_{\ell-1})) 
\to \Omega^\ell(\Sigma;\hom(\bar C_0,\bar C_{\ell}))
\]
to the form $d\cal J_{\ell-1}^\text{an}$.

\begin{defn} \label{defn:JL}
For $2\le \ell \le q-p$ we set
\begin{equation}\label{eqn:JL}
\cal J^{\text{\rm an}}_\ell := \partial^\dagger_W d\cal J_{\ell-1}^\text{an}
 \in \Omega^\ell(\Sigma;  \hom(\bar C_0,\bar C_\ell))\, .
\end{equation}
\end{defn}
 
The above defines $\cal J^{\text{\rm an}}_\ell$ for $0 \le \ell \le q-p$.
For all other $\ell$ we set $\cal J^{\text{\rm an}}_\ell = 0$.
Taking the  sum over $\ell$ then defines the form
\[
\cal J^{\text{\rm an}} = \textstyle \sum_\ell \cal J_{\ell}^\an 
\in \Omega^{\ast}(\Sigma;\End(\bar C)) \, .
\]

\begin{lem} \label{lem:continuity} For all $\ell$ we have
\[
\eth \cal J_{\ell}^\an = d\cal J_{\ell-1}^\an\, .
\]
\end{lem}

\begin{proof} The statement is obvious for $\ell=0$.
For $\ell > 0$ let $c \in \bar C_0$ be a 0-chain.
we apply the boundary operator $\partial$ for the chain complex $\bar C$
 to the equation 
$\cal J_{\ell}^\an(c) = 
\partial^\dagger_W d\cal J_{\ell-1}^\an(c)$. This results in the identity
\[
\partial \cal J_{\ell}^\an(c) = d\cal J_{\ell-1}^\an(c)\, .
\]
On the other hand 
\[
(\eth \cal J_{\ell}^\an)(c) = \partial \cal J_{\ell}^\an(c)
\]
since $\partial c = 0$.
\end{proof}

\begin{cor} \label{cor:chain-map} 
The form $\cal J^{\text{\rm an}}$ satisfies the identity
\[
\delta \cal J^{\text{\rm an}} \, = \, d\cal J_{q-p}^{\text{\rm an}}\, ,
\]
where $\delta$ is the coboundary operator for the complex
$\Omega^{\ast}(\Sigma;\End(\bar C))$.  
\end{cor}

\begin{proof} 
We have 
\begin{align*}
\delta \cal J^{\text{\rm an}} &= \sum_\ell \delta \cal J_\ell^\an\, ,\\
& = \sum_{\ell} d\cal J^{\text{\rm an}}_\ell +  (-1)^{\ell} d\cal J^{\text{\rm an}}_{\ell-1} \, ,\\
& =  d\cal J^{\text{\rm an}}_{q-p}\, ,
\end{align*}
where the second line of the display uses the definition of $\delta$ as well
as the identity  $\eth \cal J^{\text{\rm an}}_\ell = d\cal J^{\text{\rm an}}_{\ell-1}$.
\end{proof}



\begin{proof}[Proof of Theorem \ref{bigthm:uniqueness}] To prove the 
existence part of Theorem \ref{bigthm:uniqueness}  we need to show
that the forms $\cal J_\ell^{\text{\rm an}}$ that were constructed above satisfy 
 axioms A1-A3.
 Axiom A1 was established in Lemma \ref{lem:continuity}. Axiom A2 holds
 when $\ell = 0$ since $\cal J_0^{\text{\rm an}}= I - i_W^\dagger$ is
  the orthogonal projection onto $\bar D_{W0}$ in the modified inner product.
  For $\ell > 0$ axiom A2 follows from the definition
 \[
 \cal J_\ell^{\text{\rm an}} = \partial^\dagger_W d\cal J_{\ell-1}^{\text{\rm an}}
 \]
 and the formula for the pseudoinverse of $\partial$ given by
 \[
 \partial_W^\dagger := \partial^\ast_W (\partial\partial^\ast_W)^{-1}\, .
 \]
It follows that $\cal J_\ell^{\text{\rm an}}$ lies in the subspace
 $\Omega^\ell(\Sigma;\hom(\bar C_0, \bar D_{W\ell}))$.
The verification of axiom A3 is trivial.

 We next show that the axioms uniquely determine the forms $\cal J_\ell^{\text{\rm an}}$. Set $\Phi_\ell = \cal J_{\ell}^{\text{\rm an}}$ and suppose that $\Psi_\ell$ is another collection of forms satisfying the axioms.

Set $\alpha_\ell = \Phi_\ell - \Psi_\ell$. It suffices to prove that
$\alpha_\ell$ is trivial.
The forms $\alpha_\ell$ satisfy axiom A2 x the axioms
 \begin{enumerate}
 \item[(B1):] $\eth \alpha_\ell = 0$;
 \item[(B3):] $\alpha_0\: \Sigma \to \End(H_0(\bar C))$ is trivial.
 \end{enumerate}
For the argument we rename Axiom A2 by B2. 
 
For $\ell =0$ 
 axiom B1 implies
 that $\alpha_0\: \Sigma \to \End(\bar C)_0$ takes values in chain maps.
 Project $\alpha_0$ onto the $j$th component of $\End(\bar C)_0$
 to obtain a map $\alpha_{0,j}\:\Sigma \to \End(\bar C_j)$ and let $b\in \Sigma$ be
 any point. Define $f_{0,j}\: \bar C_j \to \bar C_j$ by $f_{0,j} = \alpha_{0,j}(b)$.
 
Axiom B2 shows that $f_{0,0}$
 factors as
 \[
 \bar C_0 @>>> \bar D_{W0} \subset \bar C_0
 \]
 where $\bar D_{W0}$ is the orthogonal complement of $B_0$ in the modified
 inner product. Axiom B3 implies that the first map in the composite is trivial. 
 If follows that $f_{0,0}$ is trivial. 
 A similar argument shows that $f_{0,j}$ is trivial for $j >0$. Consequently,
 $\alpha_0 = 0$.
 This establishes the basis step of the induction on $\ell$.

 For the inductive step, 
 assume
 $\alpha_j = 0$ for $0 \le j \le \ell-1$.  
 Let $f_{\ell,k}\: \Bar C_k \to \bar C_{\ell+k}$ be the $k$-th component of
 the evaluation of $\alpha_\ell$ at any vector
   $v\in  \Lambda^\ell T_b\Sigma$ for some $b\in \Sigma$. In other words,
   $\alpha_\ell(v) = \oplus_k f_{\ell,k}$. Then
 Axiom B1 implies that the diagram
 \[
 \xymatrix{
 \bar C_k \ar[r]^{f_{\ell,k}} \ar[d]_{\partial} & \bar C_{\ell+k} \ar[d]^{\partial} \\
\bar C_{k-1} \ar[r]_{f_{\ell, k-1}} & \bar C_{\ell+k-1}
}
\]
commutes up to sign. It follows that $\partial f_{\ell,k} = 
\pm f_{\ell,k-1}\partial$.
We next male subsidiary induction in $k$.
An argument like the one in the previous paragraph shows that $f_{\ell,0}$ is trivial. Assume that $f_{\ell,k-1}$ is trivial. Then by axiom B2, the image of $f_{\ell,k}$ lies in $\bar D_{W,\ell+k}\cap \bar Z_{\ell+k} = 0$. It follows that $f_{\ell,k}$ is trivial, establishing the inductive step indexed for the pair  
$(\ell,k)$. It follows that $\alpha_\ell =0$, completing the induction in $\ell$.
\end{proof}

\subsection{An explicit formula} 
Let 0-forms $\alpha_j \in\Omega^0(\Sigma;\End(\bar C)_0)$
be defined by 
\[
\alpha_0 := I - i_W^\dagger,\qquad \alpha_j \,\, :=\,\,   
\partial_W^\dagger \quad \text{ for }\,\,  0 < j \le q-p\, .
\]
Set 
\[
\beta_j := d\alpha_j\, , \qquad 0 \le j \le q-p\, .
\]
Then $\beta_j$ is a $1$-form.

\begin{cor}\label{lem:explicit-analytic} 
for $0 \le \ell \le q-p$, the following identity is satisfied:
\[
\cal J^{\an}_\ell =
\begin{cases}
\alpha_0 \,            \quad & \text{\rm if } \ell = 0\, , \\
   \alpha_{\ell} \wedge \beta_{\ell-1}\wedge\cdots \wedge 
\beta_0 &\text{\rm otherwise.}
\end{cases}
\]
\end{cor}

\begin{proof} The forms
on the right side
satisfy axioms A1-A3.  
Consequently, the result follows from
Theorem \ref{bigthm:uniqueness}.
\end{proof}

 \section{The Quantization Theorem} \label{sec:quantize}
 
 In this section we carefully formulate and prove 
Theorem \ref{bigthm:quantize}. In what follows we assume $\Bbb F = \Bbb R$.

 \subsection{An analytical  cochain}
Let 
\[
R \subset I_{\Sigma}
\] 
be a closed subset. 
The following are some examples of interest:
\begin{enumerate}
\item $R = R_\Sigma$ is the poset of {\it smooth} small singular simplices.
\item  $R = R_\Sigma^{q-p}$ is the poset of smooth small singular simplices
of dimension at most $q-p$.
\item $R = R_\sigma$ consists of a smooth small singular simplex
$\sigma \: \Delta^j \to \Sigma$ together with all of its faces, where $j \le q-p$.
\end{enumerate}

Let $R_\ast$ be the chain complex over $\Bbb R$ that is freely generated 
in each degree by the objects of $R$. Then integration along the 
elements of $R$ defines a map of
cochain complexes
\begin{align} \label{eqn:Stokes-map}
 \Omega^\ast(\Sigma;\End(\bar C)) & @>>> \hom(R_\ast,\End(\bar C))\, , \\
\notag \phi\qquad  &\mapsto \quad (\sigma \mapsto \textstyle\int_{\Delta^{j_{\sigma}}} \sigma^\ast \phi)\, ,
\end{align}
called the {\it Stokes map}.
It is a quasi-isomorphism when $R= R_{\Sigma}$.

\begin{defn}\label{defn:JT}  The {\it restricted analytical hypercurrent cochain}
\[
\cal J^\an_{|R}  \in \hom(R_\ast,\End(\bar C))
\]
is given by the image of the analytical hypercurrent current form $\cal J^\an$ with respect to the Stokes map 
\eqref{eqn:Stokes-map}.
\end{defn}

\begin{lem} \label{lem:cocycle-restrict} If $R = R_\Sigma^{q-p}$ or $R = R_\sigma$, then 
$\cal J^\an_{|R}$ is a cocycle.
\end{lem}

\begin{proof} By Corollary  \ref{cor:chain-map}, 
\[
\delta \cal J^\an =
d\cal J^\an_{q-p}  \in \Omega^{q-p+1}(\Sigma;\hom(\bar C_0,\bar C_{q-p}))\, .
\] 
The result then follows by observing that the
 pullback of $d\cal J^\an_{q-p}$
 along a small smooth singular simplex of dimension $\le q-p$ 
 is  trivial for dimensional reasons.
 \end{proof}

Let $\cal K^\ast(R)$ denote the right side of \eqref{eqn:Stokes-map}.
Then there is a canonical isomorphism of vector spaces
\[
\hom_{\text{ho\,}{\cal C}^{R}}(\chi,\chi) \cong H^0(\cal K^\ast(R))\, ,
\] 
where the left side denotes  hom taken in the homotopy category of
$\Ch^R$ (cf.~ Definition \ref{defn:posets})  and $\chi$ is the constant functor with value $\bar C$.

\begin{defn}\label{defn:lengthy} (1). Let $\cal U_R$ be the vector space of $0$-cochains of the complex
$\cal K^\ast(R)$. Let 
\[
\cal B_R \subset \cal U_R \qquad \text{ and } \qquad \cal Z_R \subset \cal U_R
\]
be the space of $0$-coboundaries and  0-cocycles respectively. 
\smallskip

\noindent (2).
Set
\[
\cal U^\sharp_R  := \cal U_R/\cal B_R\, .
\]
For an affine subspace $V\subset \cal U_R$ we let $V^\sharp \subset \cal U^\sharp_R$
be the image of $V$ under the quotient map $\cal U_R \to \cal U^\sharp_R$.
When $V= \cal Z_R$ we set
\[
\cal H_R := \cal Z_R^\sharp = \cal Z_R/\cal B_R\, .
\]
\smallskip

\noindent (3). For $\xi \in \End(H_0(\bar C))$, let
\[
\cal Z_R(\xi) \subset \cal Z_R
\]
be the subspace of 0-cocycles $c\: R_\ast \to \End(\bar C)$
such for $\sigma \in R$, the chain map $c(\sigma)\: \bar C \to \bar C$
induces $\xi$ in homology in degree 0, and is trivial in homology
in higher degrees. We set
\[
\cal H_R(\xi) := (\cal Z_R(\xi))^\sharp\, ,
\]
i.e., the image of $\cal Z_R(\xi)$ under the projection $\cal Z_R \to \cal Z_R^\sharp$. 
\smallskip

\noindent (4). Let
\[
\cal T_R \subset \cal U_R
\] be the subspace of $0$-cochains
$c\: R_\ast \to \End(\bar C)_\ast$ such that if $\sigma\in R$, then $c(\sigma)$ lies in $\hom(\bar C,\tau(\sigma))$, where
$\tau(\sigma) = \bar C(T_\sigma)$ is the value of the tree functor at 
the small simplex $\sigma$. 
\smallskip

\noindent (5). Let 
\[
\cal T_R^\sharp \subset \cal U^\sharp_R
\] denote the image of $\cal T_R$ under the quotient
$\cal U_R \to \cal U^\sharp_R$.
\end{defn}

Observe that $\cal H_R(\xi) \subset \cal U^\sharp$
is an affine subspace, whereas  $\cal T_R^\sharp\subset \cal U^\sharp$ is a vector subspace.

\begin{prop} \label{prop:exist-unique-thc-reform}
Let $R$ be any of the posets of interest and suppose $\xi\in \End(H_0(\bar C))$
is the identity. Then the intersection
\[
\cal H_R(\xi)\cap \cal T_R^\sharp
\]
consists of a single element, namely, the affine subspace 
\[
 \cal J_{|R} + \cal B_R
\]
associated with the restricted topological hypercurrent chain map $\cal J_{|R}$. 
\end{prop}

\begin{proof} An unravelling of the
definitions shows the assertion to be a reformulation of Corollary  
\ref{cor:exist-unique-variant}.
\end{proof}

\subsection{Topologizing  $\cal U_R^\sharp$} The
 affine spaces considered above are subspaces of
$\cal U^\sharp_R$, the latter which is typically infinite dimensional when
$R = R_\Sigma$ or $R^{q-p}_\Sigma$. We will topologize
 $\cal U^\sharp_R$ as a quotient of 
$\cal U_R$. We will topologize the latter as subspace of the product
\[
\cal U_R \subset \prod_{\sigma\in R} \End(\bar C)_{j_{\sigma}}\, .
\]
We are reduced to describing a topology on the displayed product.
Note that the real vector space $\End(\bar C)_{j_{\sigma}}$ is finite
dimensional. We topologize the displayed  product
using the box topology. 

Let
$\pi\: \cal U_R\to \cal U^\sharp_R$ denote the quotient map. 
The following is a trivial consequence of the definitions.

\begin{lem} Let $\hat x\: \Bbb N \to  \cal U_R$ be a sequence. 
Then the composite sequence
\[
x\: \Bbb N @>\hat x>>  \cal U_R @>\pi >> \cal U^\sharp_R
\]
converges to $y \in \cal U^\sharp_R$ if and only if
 for 
any open neighborhood $V$ of $y$, there is an 
$N \ge 0$ such that $\hat x_n$ is contained in $\hat V := \pi^{-1}(V)$ for $n > N$.
\end{lem}

\subsection{The low temperature limit}
The form ${\cal J}^\an$ depends in a crucial way on the choice of smooth 
protocol
$\gamma\: \Sigma \to \cal M_{p,q}$.  If $\beta\in \Bbb R_+$, let
$\gamma_\beta\: \Sigma \to \cal M_{p,q}$ be the the protocol given by
\[
\gamma_\beta(x) := \beta\gamma(x)\, ,
\]
where the right side is given by scalar multiplication.
By replacing $\gamma$ with $\gamma_\beta$ in the construction
of the analytical hypercurrent we similarly obtain
\[
{\cal J}^{\an,\beta} \in \Omega^\ast(\Sigma;\End(\bar C)) \, .
\]
If $R$ is one of the posets of interest, then the corresponding restricted analytical hypercurrent cochain is denoted
by 
\[
{\cal J}^{\an,\beta}_{|R}\, .
\]
By mild abuse of notation we identify the latter with its coset 
${\cal J}^{\an,\beta}_{|R} + \cal B_R$ in 
$\cal U^\sharp$.

\begin{thm}[Quantization] \label{thm:quantization} 
Let $R \subset R_{\Sigma}^{q-p}$ be closed.
Then in $\cal U^\sharp$ we have
\[
\lim_{\beta \to \infty} {\cal J}^{\an,\beta}_{|R} = {\cal J}_{|R}\, ,
\]
where the limit is taken in the topology described above.
\end{thm}

\begin{rem}
The limit  $\beta\to \infty$ is known as the {\it low temperature limit}.
Theorem \ref{thm:quantization} says that the low temperature limit of the restricted analytical hypercurrent cochain coincides with its topological counterpart.
\end{rem}

\subsection{The Kirchhoff decomposition}
In this subsection, we extend the generalized Kirchhoff and Boltzmann
distributions of \cite{CCK15a},\cite{CCK15b} to deduce a 
decomposition of the form $\cal J^\an_\ell$. This decomposition is subsequently used in the proof of the Quantization Theorem.
In what follows we fix a pair
\[
(\gamma,\beta)\, ,
\]
where $\gamma \: \Sigma \to \cal M_{p,q}(X)$ is a protocol
and $\beta\in \Bbb R_+$ is a choice of inverse
temperature $\beta \in \Bbb R_+$. Recall that $\gamma_\beta(x) := 
\beta\gamma(x )$.

Let $T$  be a $d$-tree of $X$, where $d\in [p,q]$ (cf.~Definition \ref{defn:d-tree})
If $d > p$, then
$T$  determines a preferred right inverse
\[
\partial_{T}^+ \: \bar B_{d-p-1} \to \bar C_{d-p}
\]
for the boundary operator $\partial\: \bar C_{d-p} \to \bar B_{d-p-1}$. The right inverse is obtain by inverting the composition 
\[
\bar C_{d-p}(T) @> \subset>> \bar C_{d-p}(X)  @>\partial >> \bar B_{d-p-1}(X)\, ,
\]
the latter which is an isomorphism. 

If $d=p$, then a $p$-co-tree $T$ determines a right inverse
$-i^+_{T}$ to  $-i\:\bar B_0 \to \bar C_0$, where $i$ denotes the inclusion.
In this case, it will be convenient to abuse notation slightly by
setting
\[
\partial_{T}^+ := -i^+_{T}\: \bar C_0 \to \bar B_0\, ,
\]

For a  $d$-tree  $T$ with $d > p$, let 
$
\tau_{T} 
$
be the order of the torsion subgroup of 
$H_{d-1}(T;\Bbb Z)$. When $j=p$
we define $\tau_T$ to be the order of the torsion subgroup
of $C_p(X;\Bbb Z)/B_p(T;\Bbb Z)$. Let $W_{T}\: \Sigma \to \Bbb R$ be defined by $W_{T}(x) = \sum_{b\in  T_d} W_d(x)(b)$. Set
\begin{equation}\label{eqn:rho-T}
\rho_{T} = \tfrac{1}{\Delta}\tau_T^2 e^{-\beta W_{T}}\, ,\quad   \Delta
= \sum_T \tau_T^2 e^{-\beta W_{T}}\, , 
\end{equation}
where the sum is over all $d$-trees.
 
\begin{lem}[cf.~\cite{CCK15a},\cite{CCK15b}] \label{lem:kirchhoff-sum}  
With respect to the modified inner product defined by $\beta W_j$,
the Moore-Penrose pseudoinverse to 
the boundary operator $\partial\: \bar C_j \to \bar B_{j-1}$ 
(respectively to the inclusion $-i\: \bar B_0 \to \bar C_0$ if  $j = 0$)
is given by the expression
\[
\partial_{j}^\dagger := \sum_{T} \rho_{T}\partial_{T}^+ \, ,
\]
where the sum is over all $(j+p)$-trees $T$.
\end{lem}

\begin{rem} 
The function $\rho_{T}\: \Sigma \to \Bbb R_+$
and the operator $\partial_j^\dagger$ depend on the value of
$\beta W_{j+p}$ even though the notation fails to indicate it.
The operator
$\partial_{T}^+$, which depends on
$T$, is independent of $\beta W_{j+p}$.
\end{rem}

\begin{defn}\label{defn:ell-structure} Let $\ell \in [0,q-p]$ be an integer.
 An {\it $\ell$-orchard} for $X$ is an ordered
$(\ell+1)$-tuple
\[
\omega_\bullet = (\omega_0,\omega_1,\dots,\omega_{\ell})\, ,
\]
in which $\omega_j$ is a $(j+p)$-tree of $X$.
\end{defn}

To build a Kirchhoff decomposition for the form $\cal J^\an_\ell$,
we arbitrarily fix, for each $j \in (0,\ell)$, a left inverse 
\[
\zeta_j \: \bar C_j \to \bar B_j
\] 
to the inclusion $\bar B_j \to \bar C_j$; denote these data by $\zeta_\bullet$. For an $\ell$-orchard $\omega_\bullet$, 
let
$
f_\ell(\omega_\bullet)\: \bar C_0 \to \bar C_\ell
$
be the linear transformation defined by
\begin{equation} \label{eqn:f-omega}
f_\ell(\omega_\bullet) = 
\partial_{\omega_\ell}^+ \zeta_{\ell-1} \partial_{\omega_{\ell-1}}^+
 \zeta_{\ell-2} \partial_{\omega_{\ell-2}}^+\cdots 
 \zeta_{1} \partial_{\omega_{1}}^+\partial_{\omega_{0}}^+\, .
 \end{equation}
 Define an $\ell$-form 
 $
 \varrho_\ell(\omega_\bullet) \in \Omega^{\ell}(\Sigma;\Bbb R)
 $
 by 
 \begin{equation} \label{eqn:J-omega}
  \varrho_\ell(\omega_\bullet) := \rho_{\omega_\ell}
  d\rho_{\omega_{\ell-1}}\wedge
  \cdots \wedge d\rho_{\omega_{0}}\, .
\end{equation}
Hence, $f_\ell$  depends on the pair $(\omega_\bullet,\zeta_\bullet)$,
whereas $\varrho_\ell$ depends on the triple
 $(\omega_\bullet, \gamma,\beta)$.
  
  \begin{prop}[Kirchhoff Decomposition] \label{prop:k-decomp} The form $\cal J^\an_\ell$ is given by
  \[
  \sum_{\omega_\bullet} f_\ell(\omega_\bullet) \varrho_\ell(\omega_\bullet)\, ,
  \]
  where the sum is indexed over all $\ell$-orchards of $X$.
  \end{prop}
  
   \begin{rem} In particular, Proposition \ref{prop:k-decomp} implies that the sum
   of the terms $f_\ell(\omega_\bullet) \varrho_\ell(\omega_\bullet)$
  is independent of the choice of $\zeta_\bullet$.
  \end{rem}

  \begin{proof}[Proof of Proposition \ref{prop:k-decomp}] Observe  that $\zeta_{j-1} d\partial^\dagger_{j-1}$ coincides with
  $d\partial^\dagger_{j-1}$, since the image of 
  $d\partial^\dagger_{j-1}$ is
  contained in  $\bar B_{j-1}$.    Then
  
  \resizebox{1.05\linewidth}{!}{
  \begin{minipage}{\linewidth}
  \begin{align*}
  \cal J^\an_\ell  
  &= \partial_\ell^\dagger d \partial_{\ell-1}^\dagger
  \wedge \cdots \wedge d \partial_1^\dagger \wedge d \partial_0^\dagger \, , \\
  &= \partial_\ell^\dagger \zeta_{\ell-1} d \partial_{\ell-1}^\dagger
  \wedge \cdots \wedge \zeta_1 d \partial_1^\dagger \wedge d \partial_0^\dagger \, , \\
 &=  \sum_{\omega_\ell} \rho_{\omega_\ell}\partial_{\omega_\ell}^+
\sum_{\omega_{\ell-1}} d\rho_{\omega_{\ell-1}}\zeta_{\ell-1}
 \partial_{\omega_{\ell-1}}^+ \wedge \cdots \wedge \sum_{\omega_{1}} 
 d\rho_{\omega_{1}}\zeta_{1}
 \partial_{\omega_{1}}^+ \wedge \sum_{\omega_0} \rho_{\omega_0}\partial_{\omega_0}^+ \, ,\\
  & =  \sum_{\omega_\bullet} f_\ell(\omega_\bullet) \cal \varrho_\ell(\omega_\bullet)\, ,
 \end{align*}
   \end{minipage}
}
where the first line uses Lemma
\ref{lem:explicit-analytic} and the transition from line two to line three uses Lemma  \ref{lem:kirchhoff-sum}.   \end{proof}

 \subsection{Proof of the Quantization Theorem}
With respect to Definition \ref{defn:lengthy}, set
\[
\cal D^\sharp_R :=\cal U^\sharp_R/\cal T^\sharp_R\, .
\]
Then we have an short exact sequence of vector spaces
\[
0\to \cal T^\sharp_R @> i >> \cal U^\sharp_R @>p>> \cal D^\sharp_R\to 0
\]
in which each term is finite dimensional when $
R = R_\sigma$.

\begin{lem} \label{lem:one-to-one}
For any of the posets of interest, the composition
\[
\cal H_R(0) @>>> \cal U^\sharp_R  @>p>> \cal D^\sharp_R
\]
is a monomorphism of vector spaces. 
\end{lem}

\begin{proof} The proof is almost the same as the proof of Proposition \ref{prop:exist-unique-thc-reform} where $\xi$ is
now replaced by $0$. The details will be left to the reader.
\end{proof}

The Quantization Theorem 
will follow from the next two results.

\begin{lem} \label{lem:sigma-reduction} If the Quantization Theorem holds for $R = R_\sigma$, then 
it also holds for all closed subsets of $R_{\Sigma}^{q-p}$.
\end{lem}

\begin{proof} By definition, convergence in $\cal U^\sharp_R$ is defined objectwise in $\sigma$.
\end{proof}

\begin{prop} \label{prop:sigma-reduction} Let $R = R_\sigma$. Then
\[
\lim_{\beta \to \infty} p({\cal J}^{\an,\beta}_{|R}) = 0\, .
\]
\end{prop}

 \begin{proof}[Proof of Theorem \ref{thm:quantization} assuming Proposition \ref{prop:sigma-reduction}]
 By Lemma \ref{lem:sigma-reduction}, it is enough to consider the case $R= R_\sigma$. Consider the difference
 \[
 z_\beta := {\cal J}^{\an,\beta}_{|R} - \cal J_{|R}
 \]
 Then $z_\beta \in \cal H_R(0)$ for all $\beta$. We need to show that
 $z_\beta$ tends to zero in $\cal U^\sharp_R$.
 
 Observe that $p(z_\beta) = p({\cal J}^{\an,\beta}_{|R})$.
Using Proposition \ref{prop:sigma-reduction}, we infer that 
 $z_\beta$ tends to the subspace $\cal T^\sharp_R$ as $\beta$ tends to $\infty$. Hence, $z_\beta$ tends to the intersection $\cal H_R(0)\cap \cal T^\sharp_R$. But Lemma
 \ref{lem:one-to-one} says this intersection  is just the zero vector.
It follows that $z_\beta$ tends to 0.
 \end{proof}
 
\begin{proof}[Proof of Proposition \ref{prop:sigma-reduction}]
We begin by reviewing the relevant definitions and introducing some notation.
In what follows 
\begin{itemize}
\item $\gamma\: \Sigma \to \cal M_{p,q}$ is a good protocol;
\item $\beta > 0$ is a real number;
\item For $t \in \Sigma$, and $p\le n\le q$, the function
$W_n(t)\: X_\ell\to \Bbb R$ is the $n$-th component of $\gamma(t)$.
\item  $\sigma\: \Delta^j \to \Sigma$ is a small smooth singular simplex; 
\item  $k := k(\sigma)$  is
the smallest integer such that $\sigma(\Delta^j) \subset \Sigma_k$, where
$\Sigma_k$ is the set of points $t\in \Sigma$ in which $W_k(t)\: X_k\to \Bbb R$ is one-to-one; 
\item  $T_\sigma$ is the preferred $k$-tree associated with $\sigma$; 
\item for an $n$-tree $T$, $p\le n \le q$, we define 
\[
W_T = \sum_{e\in T_n} W_e\, .
\]
\end{itemize} 
For this  proof only,  it will also be convenient to use the notation
\[
\cal J^\an(\sigma,\beta) := {\cal J}^{\an,\beta}_{|R}\, .
\]
Then we have a canonical decomposition
\begin{equation} \label{eqn:TD-decomp}
\cal J^\an(\sigma,\beta) \, =\, \cal J_T^\an(\sigma,\beta) + \cal J_D^\an(\sigma,\beta) \, ,
\end{equation}
uniquely defined by the requirement that 
\[
\cal J_T^\an(\sigma,\beta) \in \hom(\bar C_0(X),\bar C_j(T_\sigma)) \subset \hom(\bar C_0,\bar C_j)
\]
as well as the requirement that the projection (in the basis defined by the cells of $X_{p+j}$)) 
of $\cal J_D^\an(\sigma,\beta)$ onto the subspace
$\hom(\bar C_0,\bar C_j(T_\sigma))$
 is trivial. It will suffice to show that
\[
\lim_{\beta\to 0} \cal J_D^\an(\sigma,\beta) = 0\, .
\]
The proof of the latter reduces to two cases.
\medskip

\noindent {\it Case 1: $j < k$.} In this case, by the defining properties of trees,  
$\cal J^\an(\sigma,\beta)$ lies in the subspace $\hom(\bar C_0,\bar C_j(T_\sigma))$ which immediately implies that
$\cal J_D^\an(\sigma,\beta) = 0$.
\medskip

\noindent {\it Case 2: $j > k$.} Let $\omega_\bullet = (\omega_0,\dots,\omega_\ell)$ be an $\ell$-orchard (cf.~Defn.~\ref{defn:ell-structure}). Then
 one has, analogous to \eqref{eqn:TD-decomp}, the expression
\begin{equation} \label{eqn:TD-decomp-omega}
\varrho_\ell(\omega_\bullet,\beta) = \varrho^T_\ell(\omega_\bullet,\beta) + \varrho^D_\ell(\omega_\bullet,\beta)\, ,
\end{equation}
where $\varrho_\ell(\omega_\bullet,\beta)$ is defined in \eqref{eqn:J-omega} (here we are emphasizing its dependence
on the parameter $\beta$).
It will be enough to establish the following.
\smallskip

\noindent {\it Claim:} There are positive constants $C = C(\sigma), E = E(\sigma)$ such that
\[
|\varrho^D_\ell(\omega_\bullet,\beta)|  < C\beta^j e^{-\beta E}\, ,
\]
where the norm on the left is defined using the basis defined by the cells of $X$.  
\smallskip

To establish the claim, we recall the definition \eqref{eqn:J-omega}
\[
  \varrho_\ell(\omega_\bullet,\beta) := \rho_{\omega_\ell}(\beta)
  d\rho_{\omega_{\ell-1}}(\beta)\wedge
  \cdots \wedge d\rho_{\omega_{0}}(\beta)\, ,
\]
where the functions $\rho_{\omega_j}\: \Sigma \to \Bbb R$ are defined 
in \eqref{eqn:rho-T}. By direct computation, we have
\[
d\rho_{\omega_\ell}(\beta) = \beta\sum_{\alpha} \eta_{\omega_\ell}(\alpha,\beta)dW_\alpha\, ,
\]
with $\alpha$  ranging over all $(p+\ell)$-trees and
the functions $\eta_{\omega_j}(\alpha,\beta)$ satisfy the identity
\[
\eta_{\omega_\ell}(\alpha,\beta) = 
\begin{cases}  \rho_{\omega_\ell}(\beta)\rho_{\alpha}(\beta)\, , & \alpha \ne \omega_\ell\, ;\\
-\rho_{\omega_\ell}(\beta)(1-\rho_{\omega_\ell}(\beta))\, , & \alpha = \omega_\ell\, .
\end{cases}
\]
In particular, using \eqref{eqn:J-omega} we infer that
\begin{equation} \label{eqn:eta}
|\eta_{\omega_\ell}(\alpha,\beta)|\le 1
\end{equation}
for all $(p+\ell)$-trees $\alpha$.

Set
\begin{equation} \label{eqn:E}
E := \inf_{\alpha\ne T_\sigma,t\in \Delta^j} (W_\alpha(\sigma(t)) - W_{T_\sigma}(\sigma(t)))\, .
\end{equation}
Then $E >0$ by the defining property of $T_\sigma$. 

In the case when $\ell = k = k(\sigma)$, it is straightforward to check that one has a tighter bound
\[
|\eta_{\omega_k}(\alpha,\beta)| < e^{-\beta E}\, ,
\]
where in this case $\alpha$ ranges of $k$-trees distinct from $T_\sigma$. 

 Let $d\mu$ be the usual volume form for $\Delta^j$, normalized so that
 $\Delta^j$ has unit volume. Let $g\: \Delta^j\to \Bbb R$ be the unique function such that
\[
g d\mu = \sigma^\ast dW_{\omega_{j-1}} \wedge \cdots \wedge dW_{\omega_0}\, .
\]
Combining the bounds \eqref{eqn:eta} and \eqref{eqn:E} leads to the inequality
\[
|\varrho^D_\ell(\omega_\bullet,\beta)| \le A\beta^j e^{-\beta E}\int_{\Delta^j}g d\mu \, ,
\]
where $A$ is the number of $(p+j)$-trees. Setting $C = A\int_{\Delta^j}g d\mu $ establishes the claim.
\medskip

\noindent {\it Case 2: $j = k$.} This case is sufficiently similar to Case (2) and is left to the reader to verify.
\end{proof}

\section{Space level hypercurrents}\label{sec:space-level}

The goal of this section is to show that, in certain cases, the hypercurrent map
lifts to a space level construction.

\subsection{The graph case} In the graph case, a hypercurrent map
amounts to the current map in the sense of \cite{CKS13} and 
can be defined over the ring of integers. We outline the construction
using model category language.

Equip the category of topological spaces $\Top$ with 
Quillen model structure so that the weak equivalences are the
 weak homotopy equivalences, the fibrations are the Serre fibrations and the
 (Serre) cofibrations are defined by the left lifting property with respect to the acyclic
 fibrations. One can characterize the cofibrations as the (retracts of)
relative cell complexes. 

For an indexing poset $I$ whose objects have finite degree, the functor category $\Top^I$ can then be equipped with the projective model
structure, in which the weak equivalences and a fibrations are defined objectwise
and the cofibrations satisfy the relative latching condition. In this model structure
every object is fibrant.

Let $X$ be a finite connected topological graph (i.e., a CW complex of dimension one) 
and let $\Sigma$ be a good protocol for $X$. Let
$F\in \Top^{I_{\Sigma}}$ be the constant functor with value $X$
and let $\ast \in \Top^{I_{\Sigma}}$ be the constant functor with value the one-point space. 
Let $T \in \Top^{I_{\Sigma}}$ be the tree functor.

The functor $F^c\:I_{\Sigma} \to \Top$ given by
$F(\sigma) = X\times \Delta^j$ for $\sigma\: \Delta^j \to \Sigma$ is a cofibrant approximation to $F$. In particular, the lifting problem
\begin{equation} \label{eqn:lift-diagram-spaces-1}
\xymatrix{
& T \ar[d]^{\sim}\\
F^c \ar[r] \ar@{..>}[ur]
& \ast\, ,
}
\end{equation}
admits a solution (which is also unique up to contractible choice). Consequently, we have
a space level pre-hypercurrent map
\[
\frak j\: F^c \to T\, .
\]
Consider the composition
\begin{equation} \label{eqn:map-space-level-1}
F^c @> \frak j >> T @> \subset >> F\, .
\end{equation}
Taking colimits, we obtain a weak map of spaces
\begin{equation}
\frak J\: \Sigma \times X \to  X\, .
\end{equation}
defining a space level version hypercurrent map. A geometric variant of this construction
appeared in \cite[defn.~5.8]{CKS13}.

\subsection{The case $p=0,q>1$} 
In this case, $X$ is a finite connected CW complex with gap $[0,q]$ 
for $q >0$. Without loss in generality we can assume $X$ has dimension $q$.

Equip $\Top$ 
with the localization model category structure 
defined by the rationalization functor $X\mapsto X_{\Bbb Q}$ 
\cite[th.~4.7]{Barwick}
(cf.~\cite[ch.~2]{Hirschhorn} in the case of based spaces). 
If $f\: X \to Y$ is a map of nilpotent spaces, then
$f$ is a weak equivalence with respect to this model structure if and only if $f$ induces an isomorphism
on homology with rational coefficients. A map $X\to Y$ is a cofibration if and only if it is a Serre cofibration.
A fibration is a map that satisfies the right lifting property with respect to the maps
which are both Serre cofibrations and weak equivalences. An object $Y$ is fibrant if and only if the rationalization map $Y \to Y_{\Bbb Q}$ is a weak homotopy equivalence, i.e.,
$Y$ is a rational space.

The construction of a space level hypercurrent map parallels the $q = 1$ case, with one difference: the tree functor is no longer fibrant. We should therefore replace it with a fibrant approximation $T^f$ given by its rationalization $T^f(s) = T(s)_{\Bbb Q}$.  In what follows $F$ is the constant functor with value $X$.
As in the $q=1$ case, we obtain a space level pre-hypercurrent map
\[
\frak j\: F^c \to T^f\, .
\] 
The space level hypercurrent map is associated with the composition
\begin{equation} \label{eqn:map-space-level-2}
F^c @> \frak j >> T^f @> \subset >> F^f\, ,
\end{equation}
where $F^f$ is the constant functor with value $X_{\Bbb Q}$.
Applying the homotopy colimit construction to the each functor of 
\eqref{eqn:map-space-level-2} and composing, we obtain a map
\[
\Sigma \times X \to \Sigma \times X_{\Bbb Q}\, .
\]
Composing the latter with the second factor projection $\Sigma \times X_{\Bbb Q} \to X_{\Bbb Q}$ gives a space level hypercurrent map
\[
\frak J\: \Sigma \times X \to  X_{\Bbb Q}\, .
\]

\subsection{The case $p > 0$} In this case, there is generally an obstruction to defining
a space level pre-hypercurrent map. However, as we will see, we can perform the construction certain cases.

Let $X$ be a finite connected CW complex $X$ having gap $[p,q]$ with $p > 1$. Without loss
in generality, we can assume that $X$ has been $(p,q)$-adjusted, i.e., $X = X^q_p$. In particular $X$ is a based space. In this
case we use in the localization model category structure on $\Top_\ast$ defined by
rationalization. 

Set $A: = H_p(X;\Bbb Q)$ and let $M := M(A,p)$ be the Moore space 
having reduced homology $A$ in degree $p$ and trivial otherwise. 

\begin{lem} If either 
\begin{itemize}
\item $p <2q$, or 
\item $p,q$ are both odd,
\end{itemize}
there is 
a based map $X\to M$ which induces an identity map in rational homology in degree $p$.
\end{lem}

\begin{proof} There is already a preferred map $X^{(p)} \to M$ which induces a projection
on rational homology in degree $p$. The obstructions to extending this map to $X$ lie
in the cohomology group $H^{q}(X/X^{(p)};\pi_{q-1}(M))$. If either of the conditions hold, then $\pi_{q-1}(M)$ is trivial.
\end{proof}
 
In what follows, we fix such a map $X\to M$. To construct a space level pre-hypercurrent map, we let $F$ be the constant functor with value $X$ and we let
$H$ be the constant functor with value $M$. Using the model category factorization axioms, the natural transformations $T\to F \to H$ fit into a commutative diagram
\[
\xymatrix{
T \ar[d]\ar[r]^{{}_\sim} & T'\ar[d]  \ar[r] &H \ar@{=}[d] \\
F \ar[r]_{{}_\sim}       & F' \ar[r] &    H
}
\]
in which the left horizontal maps are acyclic cofibrations and the right maps
are fibrations. It is automatic that the map $T'\to H$ is a weak equivalence.
It is also automatic that both $T'$ and $F'$ are fibrant since $H$ is.

In particular, the lifting problem
\begin{equation} \label{eqn:lift-diagram-spaces-2}
\xymatrix{
& T' \ar[d]^{{}_\sim}\\
F^c \ar[r] \ar@{..>}[ur]
& H\, ,
}
\end{equation}
admits a solution $F^c \to T'$, where $F^c$ is the functor 
which assigns to a small simplex $\Delta^j \to \Sigma$ the space $\Delta^j\times X$.
Taking homotopy colimits of the composition
\[
F^c  \to T' \to F'
\]
we obtain a map
\[
\Sigma \times X \to \Sigma \times X_{\Bbb Q}\, .
\]  Composing with the 
second factor projection, we obtain a space level hypercurrent map
\[
\frak J\: \Sigma \times X \to X_{\Bbb Q}\, .
\]


\end{document}